\documentclass[a4paper,10pt,twoside]{article}
\usepackage[cp1250]{inputenc}
\usepackage{amsfonts}
\usepackage{amsmath}
\usepackage{amssymb}
\usepackage{amsthm}
\usepackage{authblk}
\usepackage{blindtext}

\title{\textbf{On the uniform convergence of double sine series}}
\author[1]{\textbf{K.~Duzinkiewicz}}
\author[2]{\textbf{B.~Szal}}
\affil[ ]{University of Zielona G\'ora}
\affil[ ]{Faculty of Mathematics, Computer Science and Econometrics}
\affil[ ]{65-516 Zielona G\'ora, ul. Szafrana 4a, Poland}
\affil[1]{\tt K.Duzinkiewicz@wmie.uz.zgora.pl}
\affil[2]{\tt B.Szal@wmie.uz.zgora.pl}

\date{}

\numberwithin{equation}{section}

\usepackage{fancyhdr}
\fancypagestyle{plain}{\fancyhead{}
} 

\theoremstyle{plain}
\newtheorem{twr}{Theorem}
\newtheorem{lem}{Lemma}
\newtheorem{df}{Definition}
\newtheorem{rem}{Remark}
\newtheorem{cor}{Corollary}

\begin{document}
\maketitle
\begin{abstract}
\noindent The fundamental theorem in the theory of the uniform convergence of sine series is due to Chaundy and Jolliffe from 1916 (see \cite{CH&J}). 
Several authors gave conditions for this problem supposing that coefficients are monotone, non-negative or more recently, general monotone (see \cite{T}, \cite{S} and \cite{K}, for example). There are also results for the regular convergence of double sine series to by uniform in case the coefficients are monotone or general monotone double sequences. In this article we give new sufficient conditions for the uniformity of the regular convergence of double sine series, which are necessary as well in case the coefficients are non-negative. We shall generalize those results defining a new class of double sequences for the coefficients.
\end{abstract}

\section{Known results: uniform convergence of single sine series}

Let $\{c_{k}\}_{k=1}^{\infty}$ be a non-negative real sequence and consider
the series

\begin{equation}
\sum_{k=1}^{\infty }{c_{k}}\sin kx.  \label{1.1}
\end{equation}

In 1916 Chaundy and Jolliffe \cite{CH&J} proved the following classical
result.

\begin{twr}
If $\{c_{k}\}_{k=1}^{\infty }\subset \mathbb{R}_{+}$ is decreasing to zero,
then (\ref{1.1}) converges uniformly in $x$ if and only if 
\begin{equation}
kc_{k}\rightarrow 0\qquad \text{as}\qquad k\rightarrow \infty .  \label{1.2}
\end{equation}%
\label{T.1}
\end{twr}

Several classes of sequences have been introduced to generalize Theorem \ref%
{T.1} (see \cite{Z&Z&Y}, \cite{Y&Z}, \cite{K}, \cite{S}). These classes are
larger than the class of monotone sequences and contain sequences of complex
numbers as well. The definitions of the latest classes are the following:

\begin{align*}
& MVBVS=\Bigg\{\{c_{k}\}_{k=1}^{\infty }\subset \mathbb{C},\exists \mathcal{C%
}>0,\lambda \geq 2\;:\;\sum_{k=n}^{2n}|\Delta _{1}c_{k}|\leq \frac{\mathcal{C%
}}{n}\sum_{k=\lceil n/\lambda \rceil }^{\lceil \lambda n\rceil
}|c_{k}|,n\geq \lambda \Bigg\}, \\
& SBVS=\Bigg\{\{c_{k}\}_{k=1}^{\infty }\subset \mathbb{C},\exists \mathcal{C}%
>0,\lambda \geq 2\;:\;\sum_{k=n}^{2n-1}|\Delta _{1}c_{k}|\leq \frac{\mathcal{%
C}}{n}\left( \sup_{m\geq \lceil n/\lambda \rceil
}\sum_{k=m}^{2m}|c_{k}|\right) ,n\geq \lambda \Bigg\}, \\
& SBVS_{2}=\Bigg\{\{c_{k}\}_{k=1}^{\infty }\subset \mathbb{C},\exists 
\mathcal{C}>0,b(k)\nearrow \infty \;:\;\sum_{k=n}^{2n-1}|\Delta
_{1}c_{k}|\leq \frac{\mathcal{C}}{n}\left( \sup_{m\geq
b(n)}\sum_{k=m}^{2m}|c_{k}|\right) ,n\geq 1\Bigg\}, \\
& GM(\beta ,r)=\Bigg\{\{c_{k}\}_{k=1}^{\infty }\subset \mathbb{C},\exists 
\mathcal{C}>0\;:\;\sum_{k=n}^{2n-1}|\Delta _{r}a_{k}|\leq \mathcal{C}\beta
_{n},n\geq 1\Bigg\},
\end{align*}%
where $\Delta _{r}c_{k}=c_{k}-c_{k+r}$ for $r\in \mathbb{N}$, and the
constants $\mathcal{C}$ and $\lambda $ depend only on $\{c_{k}\}$, a
sequence $\{b(k)\}_{k=1}^{\infty }\subset \mathbb{R}_{+}$ is increasing and $%
\beta :=(\beta _{k})$ is a non-negative sequence. It was proved in \cite{K}
that $MVBVS\subsetneqq SBVS\subsetneqq SBVS_{2}$ and a series (\ref{1.1})
with coefficients of complex numbers from the classes $MVBVS$, $SBVS$, $%
SBVS_{2}$, is uniformly convergent if (\ref{1.2}) is satisfied. In \cite{S}
Szal showed that the class $GM(\beta ^{\ast },r)$, with $\beta ^{\ast
}=\beta _{n}=\frac{1}{n}\sum_{k=\lceil n/\lambda \rceil }^{\lceil \lambda
n\rceil }|c_{k}|$ ($\lambda >1$) and $r=2$ is larger than $MVBVS$ and a
series (\ref{1.1}) with coefficients from $GM(\beta ^{\ast },2)$ is
uniformly convergent if (\ref{1.2}) holds. However, the necessity of
condition (\ref{1.2}) for the uniform convergence of (\ref{1.1}) is proved
for sine series with coefficients of non-negative numbers from one of the
above classes.

\section{Known results: uniform convergence of double sine series}

We start this section by giving some definitions and notations.\newline

A double series 
\begin{equation*}
\sum_{j=1}^{\infty }\sum_{k=1}^{\infty }c_{jk}
\end{equation*}%
of complex numbers converge regularly if the sum 
\begin{equation*}
\sum_{j=1}^{m}\sum_{k=1}^{n}c_{jk}
\end{equation*}%
converge to a finite number as $m$ and $n$ tend to infinity independently of
each other, moreover, both the column series and row series 
\begin{equation*}
\sum_{j=1}^{\infty }z_{jn},\qquad n=1,2,\ldots ,\qquad \text{and}\qquad
\sum_{k=1}^{\infty }z_{mk}\qquad m=1,2,\ldots
\end{equation*}%
are convergent. Or equivalently, if for any $\epsilon >0$ there exists a
positive number $m_{0}=m_{0}(\epsilon )$ such that 
\begin{equation*}
\left\vert \sum_{j=m}^{M}\sum_{k=n}^{N}z_{jk}\right\vert <\epsilon
\end{equation*}%
holds for any $m,n,M,N$ for which $m+n>m_{0}$, $1\leq m\leq M$ and $1\leq
n\leq N$.

A monotonically decreasing double sequence $\{c_{jk}\}_{j,k=1}^{\infty }$ is
a sequence of real numbers such that 
\begin{equation*}
\Delta _{10}c_{jk}\geq 0,\qquad \Delta _{01}c_{jk}\geq 0,\qquad \Delta
_{11}c_{jk}\geq 0,\qquad j,k=1,2,\ldots ,
\end{equation*}%
where 
\begin{align*}
& \Delta _{10}c_{jk}:=c_{jk}-c_{j+1,k},\qquad \Delta
_{01}c_{jk}:=c_{jk}-c_{j,k+1}, \\
& \Delta _{11}c_{jk}:=\Delta _{10}(\Delta _{01}c_{jk})=\Delta _{01}(\Delta
_{10}c_{jk})=c_{jk}-c_{j+1,k}-c_{j,k+1}+c_{j+1,k+1}.
\end{align*}

Let $\{c_{jk}\}_{j,k=1}^{\infty }$ be a double sequence of complex numbers.
Consider the double sine series 
\begin{equation}
\sum_{j=1}^{\infty }\sum_{k=1}^{\infty }c_{jk}\sin jx\sin ky.  \label{2.1}
\end{equation}

The two-dimensional extension of Theorem \ref{T.1} is due to \v Zak and \v
Sneider.

\begin{twr}
(\cite{Z&S}) If $\{c_{jk}\}_{j,k=1}^{\infty }\subset \mathbb{R}_{+}$ is a
monotonicity decreasing double sequence, then (\ref{2.1}) is uniformly
regularly convergent in $(x,y)$ if and only if 
\begin{equation}
jkc_{jk}\rightarrow 0\qquad \text{as}\qquad j+k\rightarrow \infty .
\label{2.2}
\end{equation}%
\label{T.2}
\end{twr}

Theorem \ref{T.2} was generalized by K\'{o}rus and M\'{o}ricz \cite{2K&M}
and by K\'{o}rus \cite{2K} (and also by Leindler \cite{L}). They have defined new classes of double
sequences to obtain those generalizations. We present below those
definitions and their results.

\begin{df}
(\cite{2K&M}) A double sequence $\{c_{jk}\}_{j,k=1}^{\infty }\subset \mathbb{%
C}$ belongs to the class $MVBVDS$, if there exist positive constants $%
\mathcal{C}$ and $\lambda \geq 2$, depending only on $\{c_{jk}\}$, such
that: 
\begin{align}
\sum_{j=m}^{2m-1}{|\Delta _{10}c_{jn}|}\leq & \frac{\mathcal{C}}{m}%
\sum_{\lceil j=m/\lambda \rceil }^{\lceil \lambda m\rceil }|{c_{jn}}|,\
m\geq \lambda ,\ n\geq 1,  \label{2.3} \\
\sum_{k=n}^{2n-1}{|\Delta _{01}c_{km}|}\leq & \frac{\mathcal{C}}{n}%
\sum_{\lceil k=n/\lambda \rceil }^{\lceil \lambda n\rceil }|{c_{km}}|,\
n\geq \lambda ,\ m\geq 1,  \label{2.4} \\
\sum_{j=m}^{2m-1}\sum_{k=n}^{2n-1}{|\Delta _{11}c_{jk}|}\leq & \frac{%
\mathcal{C}}{mn}\sum_{\lceil j=m/\lambda \rceil }^{\lceil \lambda m\rceil
}\sum_{\lceil k=n/\lambda \rceil }^{\lceil \lambda n\rceil }|{c_{jk}}|,\
m,n\geq \lambda \mathbf{.}  \label{2.5}
\end{align}
\end{df}

\begin{twr}
(\cite{2K&M})

\begin{description}
\item[(i)] If $\{c_{jk}\}_{j,k=1}^{\infty}\subset\mathbb{C}$ belongs to the
class $MVBVDS$ and (\ref{2.2}) holds, then (\ref{2.1}) converges regularly,
uniformly in $(x,y)$.

\item[(ii)] Conversely, if $\{c_{jk}\}_{j,k=1}^{\infty}\subset\mathbb{R}_+$
belongs to the class $MVBVDS$ and (\ref{2.1}) is uniformly regularly
convergent in $(x,y)$, then (\ref{2.2}) is satisfied.
\end{description}

\label{T.3}
\end{twr}

\begin{df}
(\cite{2K}) A double sequence $\{c_{jk}\}_{j,k=1}^{\infty }\subset \mathbb{C}
$ belongs to the class $SBVDS_{1}$, if there exist positive constants $%
\mathcal{C}$ and integer $\lambda \geq 2$ and sequences $\{b_{1}(l)%
\}_{l=1}^{\infty }$, $\{b_{2}(l)\}_{l=1}^{\infty }$, $\{b_{3}(l)\}_{l=1}^{%
\infty }$, each one tends (not necessarily monotonically) to infinity, all
of them depending only on $\{c_{jk}\}$, such that: 
\begin{align}
\sum_{j=m}^{2m-1}{|\Delta _{10}c_{jn}|}\leq & \frac{\mathcal{C}}{m}\left(
\max_{b_{1}(m)\leq M\leq \lambda b_{1}(m)}\sum_{j=M}^{2M}|{c_{jn}}|\right)
,\ m\geq \lambda ,\ n\geq 1,  \label{2.6} \\
\sum_{k=n}^{2n-1}{|\Delta _{01}c_{km}|}\leq & \frac{\mathcal{C}}{n}\left(
\max_{b_{2}(n)\leq M\leq \lambda b_{2}(n)}\sum_{k=N}^{2N}|{c_{km}}|\right)
,\ n\geq \lambda ,\ m\geq 1,  \label{2.7} \\
\sum_{j=m}^{2m-1}\sum_{k=n}^{2n-1}{|\Delta _{11}c_{jk}|}\leq & \frac{%
\mathcal{C}}{mn}\left( \sup_{{M+N}\geq
b_{3}(m+n)}\sum_{j=M}^{2M}\sum_{k=N}^{2N}{|c_{jk}|}\right) ,\ m,n\geq
\lambda \mathbf{.}  \label{2.8}
\end{align}
\end{df}

\begin{twr}
(\cite{2K}) If $\{c_{jk}\}_{j,k=1}^\infty\subset\mathbb{R_+}$ belongs to the
class $SBVDS_1$ and (\ref{2.1}) is uniform regularly convergent in $(x,y)$,
then (\ref{2.2}) is satisfied.\label{T.4}
\end{twr}

\begin{df}
(\cite{2K}) A double sequence $\{c_{jk}\}_{j,k=1}^{\infty }\subset \mathbb{C}
$ belongs to the class $SBVDS_{2}$, if there exist positive constants $%
\mathcal{C}$ and integer $\lambda \geq 1$ and sequence $\{b(l)\}_{l=1}^{%
\infty }$ tending monotonically to infinity, depending only on $\{c_{jk}\}$,
for which: 
\begin{align}
\sum_{j=m}^{2m-1}{|\Delta _{10}c_{jn}|}\leq & \frac{\mathcal{C}}{m}\left(
\sup_{M\geq b(m)}\sum_{j=M}^{2M}|{c_{jn}}|\right) ,\ m\geq \lambda ,\ n\geq
1,  \label{2.9} \\
\sum_{k=n}^{2n-1}{|\Delta _{01}c_{km}|}\leq & \frac{\mathcal{C}}{n}\left(
\sup_{N\geq b(n)}\sum_{k=N}^{2N}|{c_{km}}|\right) ,\ n\geq \lambda ,\ m\geq
1,  \label{2.10} \\
\sum_{j=m}^{2m-1}\sum_{k=n}^{2n-1}{|\Delta _{11}c_{jk}|}\leq & \frac{%
\mathcal{C}}{mn}\left( \sup_{{M+N}\geq b(m+n)}\sum_{j=M}^{2M}\sum_{k=N}^{2N}{%
|c_{jk}|}\right) ,\ m,n\geq \lambda \mathbf{.}  \label{2.11}
\end{align}
\end{df}

\begin{twr}
(\cite{2K}) If $\{c_{jk}\}_{j,k=1}^{\infty }\subset \mathbb{C}$ belongs to
the class $SBVDS_{2}$ and (\ref{2.2}) holds, then the regularly convergent
of (\ref{2.1}) is uniform in $(x,y)$.\label{T.5}
\end{twr}

\begin{twr}
(\cite{2K}) $MVBVDS\subsetneqq SBVDS_{1}\subsetneqq SBVDS_{2}.$
\end{twr}

Now, we shall define a new class of double sequences in the following way:
Let, for $r\in \mathbb{N}$, 
\begin{equation*}
\Delta _{r0}c_{jk} :=c_{jk}-c_{j+r,k},\Delta _{0r}c_{jk} :=c_{jk}-c_{j,k+r}
\end{equation*}
and 
\begin{equation*}
\Delta _{rr}c_{jk}:=\Delta _{r0}(\Delta _{0r}c_{jk}).
\end{equation*}

\begin{df}
A double sequence $\{c_{jk}\}_{j,k=1}^{\infty}\subset \mathbb{C} $ belongs
to the class $DGM(\alpha,\beta,\gamma,r)$ (called Double General Monotone),
if there exist positive constants $\mathcal{C}$ and integer $\lambda$
depending only on $\{c_{jk}\}$, for which:

\begin{align*}
& \sum_{j=m}^{2m-1}{|\Delta _{r0}c_{jn}|}\leq \mathcal{C}\alpha _{mn},\text{ 
}m\geq \lambda ,\ n\geq 1, \\
& \sum_{k=n}^{2n-1}{|\Delta _{0r}c_{km}|}\leq \mathcal{C}\beta _{mn},\text{ }%
n\geq \lambda ,\ m\geq 1, \\
& \sum_{j=m}^{2m-1}\sum_{k=n}^{2n-1}{|\Delta _{rr}c_{jk}|}\leq \mathcal{C}%
\gamma _{mn},\text{ }m,n\geq \lambda \mathbf{.}
\end{align*}%
hold, where $\alpha :=\{\alpha _{mn}\}_{m,n=1}^{\infty }$, $\beta :=\{\beta
_{mn}\}_{m,n=1}^{\infty }$, $\gamma :=\{\gamma _{mn}\}_{m,n=1}^{\infty }$
are non-negative double sequences and $r\in \mathbb{N}$.
\end{df}

Using our definition for $r=1$, we have:

\begin{description}
\item[1)] $MVBVDS\equiv DGM(_1\alpha,_1\beta,_1\gamma,1)$, where $%
\{_1\alpha\}, \{_1\beta\}\; \text{and}\;\{_1\gamma\}$ are the sequences
defined by the formulas on the right sides of the inequalities (\ref{2.3}), (%
\ref{2.4}) and (\ref{2.5}), respectively;

\item[2)] $SBVDS_1\equiv DGM(_2\alpha,_2\beta,_2\gamma,1)$, where $%
\{_2\alpha\}, \{_2\beta\}\; \text{and}\;\{_2\gamma\}$ are the sequences
defined by the formulas on the right sides of the inequalities (\ref{2.6}), (%
\ref{2.7}) and (\ref{2.8}), respectively;

\item[3)] $SBVDS_{2}\equiv DGM(_{3}\alpha ,_{3}\beta ,_{3}\gamma ,1)$, where 
$\{_{3}\alpha \},\{_{3}\beta \}\;\text{and}\;\{_{3}\gamma \}$ are the
sequences defined by the formulas on the right sides of the inequalities (%
\ref{2.9}), (\ref{2.10}) and (\ref{2.11}); respectively.
\end{description}

In this paper we shall present some properties of the classes $%
DGM(_{2}\alpha ,_{2}\beta ,_{2}\gamma ,2)$ and $DGM(_{3}\alpha ,_{3}\beta
,_{3}\gamma ,2)$. Moreover, we generalize and extend the results of K\'{o}%
rus (\cite{2K}) to the classes $DGM(_{2}\alpha ,_{2}\beta ,_{2}\gamma ,2)$
or $DGM(_{3}\alpha ,_{3}\beta ,_{3}\gamma ,2)$, respectively.

\section{Auxiliary results}


\begin{lem}
If $\{c_{jk}\}\subset \mathbb{C}$ is such that condition (\ref{2.2}) and the
inequality 
\begin{equation}
\sum_{j=m}^{2m-1}\sum_{k=n}^{2n-1}{|\Delta _{22}c_{jk}|}\leq \frac{\mathcal{C%
}}{mn}\left( \sup_{{M+N}\geq b(m+n)}\sum_{j=M}^{2M}\sum_{k=N}^{2N}{|c_{jk}|}%
\right) ,\ m,n\geq \lambda  \label{3.1}
\end{equation}%
are satisfied, then 
\begin{equation*}
mn\sum_{j=m}^{\infty }\sum_{k=n}^{\infty }|\Delta _{22}c_{jk}|\rightarrow
0\qquad \text{as}\qquad m+n\rightarrow \infty \qquad m,n\geq \lambda \mathbf{%
.}
\end{equation*}%
\label{L.1}
\end{lem}

\begin{proof}
Set $\epsilon>0$ arbitrarily. By condition (\ref{2.2}) and from the fact that $\{b(l)\}$ tends monotonically to infinity, there exists an $m_1=m_1(\epsilon)$ such that
\begin{align}
jk|c_{jk}|<\epsilon\qquad\textrm{for all}\quad j,k\qquad j+k>b(m_1).\nonumber
\end{align}
Then, by (\ref{3.1}), assuming $m+n>m_1$ and $m,n\ge\lambda$, we have
\begin{align}
\sum_{j=m}^\infty\sum_{k=n}^\infty|\Delta_{22}c_{jk}|&=\sum_{r=0}^\infty\sum_{s=0}^\infty\sum_{j=2^rm}^{2^{r+1}m-1}\sum_{k=2^sn}^{2^{s+1}n-1}|\Delta_{22}c_{jk}|\le\sum_{r=0}^\infty\sum_{s=0}^\infty\frac{\mathcal{C}}{2^rm 2^sn}\left(\sup_{M+N\ge b(2^rm+2^sn)}\sum_{j=M}^{2M}\sum_{k=N}^{2N}|c_{jk}|\right)\nonumber\\
&<\sum_{r=0}^\infty\sum_{s=0}^\infty\frac{\mathcal{C}\epsilon}{2^rm 2^sn}\left(\sup_{M+N\ge b(2^rm+2^sn)}\sum_{j=M}^{2M}\sum_{k=N}^{2N}\frac{1}{jk}\right)\le\sum_{r=0}^\infty\sum_{s=0}^\infty\frac{4\mathcal{C}\epsilon}{2^r m 2^s n}\le\frac{16\mathcal{C}\epsilon}{mn},\label{3.2}
\end{align}
since
\begin{align*}
\sum_{j=M}^{2M}\sum_{k=N}^{2N}\frac{1}{jk}\le\sum_{j=M}^{2M}\sum_{k=N}^{2N}\frac{1}{MN}\le\frac{2M\cdot 2N}{MN}=4\qquad\textrm{for any}\qquad M,N\textbf{.}
\end{align*}
This complete the proof.
\end{proof}


\begin{lem}
Let $\{c_{jk}\}_{j,k=1}^{\infty}\subset\mathbb{C}$ and the condition (\ref%
{2.2}) is satisfied.

\begin{description}
\item[(i)] If the inequality 
\begin{equation}
\sum_{j=m}^{2m-1}{|\Delta _{20}c_{jn}|}\leq \frac{\mathcal{C}}{m}\left(
\sup_{M\geq b(m)}\sum_{j=M}^{2M}|{c_{jn}}|\right) ,\ m\geq \lambda ,\ n\geq 1
\label{3.3}
\end{equation}%
holds, then 
\begin{equation}
m\sup_{k\geq n}k\sum_{j=m}^{\infty }|\Delta _{20}c_{jk}|\rightarrow 0\mathbf{%
,}  \label{3.4}
\end{equation}

\item[(ii)] and if the inequality 
\begin{equation}
\sum_{k=n}^{2n-1}{|\Delta _{02}c_{km}|}\leq \frac{\mathcal{C}}{n}\left(
\sup_{N\geq b(n)}\sum_{k=N}^{2N}|{c_{km}}|\right) ,\ n\geq \lambda ,\ m\geq 1
\label{3.5}
\end{equation}%
holds, then 
\begin{equation}
n\sup_{j\geq m}j\sum_{k=n}^{\infty }|\Delta _{02}c_{jk}|\rightarrow 0
\label{3.6}
\end{equation}
\end{description}

$\text{as}\; m+n\rightarrow\infty,\; \text{where}\; m\ge\lambda$ and $n\ge1$
or $n\ge\lambda$ and $m\ge1$, respectively. \label{L.2}
\end{lem}

\begin{proof} Set $\epsilon>0$ arbitrarily. By condition (\ref{2.2}) and from the fact that $\{b(l)\}$ tends monotonically to infinity, there exists an $m_1=m_1(\epsilon)$ such that
\begin{align}
jk|c_{jk}|<\epsilon\qquad\textrm{for all}\quad j,k\qquad j+k>b(m_1).\nonumber
\end{align}
\begin{description}
\item[Part (i)] By (\ref{3.3}), assuming $m>\max \left\{ m_{1},\lambda \right\}$ and $n>m_1$, we have
\begin{align*}
\sup_{k\ge n}k\sum_{j=m}^\infty|\Delta_{20}c_{jk}|&=\sup_{k\ge n}k\sum_{r=0}^\infty\sum_{j=2^r m}^{2^{r+1}m-1}|\Delta_{20}c_{jk}|\le\sup_{k\ge n}k\sum_{r=0}^\infty\frac{\mathcal{C}}{2^r m}\left(\sup_{M\ge b(2^r m)}\sum_{j=M}^{2M}|c_{jk}|\right)\\
&\le\sup_{k\ge n}k\frac{\mathcal{C}}{m}\sum_{r=0}^\infty\frac{1}{2^r}\left(\sup_{M\ge b(2^r m)}\sum_{j=M}^{2M}jk|c_{jk}|\frac{1}{jk}\right)<\frac{\mathcal{C}\epsilon}{m}\sum_{r=0}^{\infty}\frac{1}{2^r}\left(\sup_{M\ge b(2^r m)}\sum_{j=M}^{2M}\frac{1}{j}\right)\le\frac{4\mathcal{C}\epsilon}{m},
\end{align*} 
since
\begin{align*}
\sum_{j=M}^{2M}\frac{1}{j}\le\frac{M+1}{M}\le 2.
\end{align*}
This implies that (\ref{3.4}) holds.

\item[Part (ii)] Using (\ref{3.5}) and assuming $n>\max \left\{ m_{1},\lambda \right\}$ and $m>m_1$, we get
\begin{align*}
\sup_{m\ge j}j\sum_{k=n}^\infty|\Delta_{02}c_{jk}|&=\sup_{m\ge j}j\sum_{r=0}^\infty\sum_{k=2^r n}^{2^{r+1}n-1}|\Delta_{02}c_{jk}|\le\sup_{m\ge j}j\sum_{r=0}^\infty\frac{\mathcal{C}}{2^r n}\left(\sup_{N\ge b(2^r n)}\sum_{k=N}^{2N}|c_{jk}|\right)\\
&\le\sup_{m\ge j}j\frac{\mathcal{C}}{n}\sum_{r=0}^\infty\frac{1}{2^r}\left(\sup_{N\ge b(2^r n)}\sum_{k=N}^{2N}jk|c_{jk}|\frac{1}{jk}\right)<\frac{\mathcal{C}\epsilon}{n}\sum_{r=0}^{\infty}\frac{1}{2^r}\left(\sup_{N\ge b(2^r n)}\sum_{k=N}^{2N}\frac{1}{k}\right)\le\frac{4\mathcal{C}\epsilon}{n}.
\end{align*}
Hence (\ref{3.6}) is satisfied.

\end{description}
Now, our proof is complete.
\end{proof}


\begin{lem}
If $\{c_{jk}\}_{j,k=1}^\infty$ is a non-negative sequence belonging to the
class $DGM(_2\alpha,_2\beta,_2\gamma,2)$ with $\mathcal{C}$, $\lambda$ and $%
\{b_1(l)\}_{l=1}^{\infty}$, $\{b_2(l)\}_{l=1}^{\infty}$, $%
\{b_3(l)\}_{l=1}^{\infty}$ then for any $m,n\ge\lambda$ 
\begin{align*}
mnc_{mn}&\le\mathcal{C}\left(\sup_{{M+N}\ge b_3(m+n)}
\sum_{j=M}^{2M}\sum_{k=N}^{2N}c_{jk}\right)+2\mathcal{C}\sum_{j=b_1(m)}^{2%
\lambda b_1(m)}\sum_{k=n}^{2n+1}c_{jk}+2\mathcal{C}\sum_{j=m}^{2m+1}%
\sum_{k=b_2(n)}^{2\lambda
b_2(n)}c_{jk}+8\sum_{j=m}^{2m+1}\sum_{k=n}^{2n+1}c_{jk}\mathbf{.}
\end{align*}
\label{L.3}
\end{lem}

\begin{proof}
Let $m,n\geq \lambda $. If $\{c_{jk}\}_{j,k=1}^{\infty }\in DGM(_{2}\alpha
,_{2}\beta ,_{2}\gamma ,2)$, then for any $v$ and $m\le\mu\le 2m$
\begin{align}
c_{mv}&=\sum_{j=m}^{\mu-1}\Delta_{20}c_{jv}+c_{\mu v}+c_{\mu+1,v}-c_{m+1,v}\le
\sum_{j=m}^{2m-1}|\Delta_{20}c_{jv}|+c_{\mu v}+c_{\mu+1,v}\nonumber\\
&\le \frac{\mathcal{C}}{m}\left(\max_{b_1(m)\le M\le\lambda b_1(m)}\sum_{j=M}^{2M}c_{jv}\right)+c_{\mu v}+c_{\mu+1,v}\textbf{.}\label{3.7}
\end{align}
By an analogous argument, we get that for any $\mu$ and $n\le v\le 2n$
\begin{align}
c_{\mu n}&=\sum_{k=n}^{v-1}\Delta_{02}c_{\mu k}+c_{\mu v}+c_{\mu,v+1}-c_{\mu,n+1}\le\sum_{k=n}^{2n-1}|\Delta_{02}c_{\mu k}|+c_{\mu v}+c_{\mu,v+1}\nonumber\\
&\le\frac{\mathcal{C}}{n}\left(\max_{b_2(n)\le N\le\lambda b_2(n)}\sum_{k=N}^{2N}c_{\mu k}\right)+c_{\mu v}+c_{\mu,v+1}\textbf{.}\label{3.8}
\end{align}
For $\mu$, $v$ such that $m\le\mu\le 2m$ and $n\le v\le 2n$ we have
\begin{align*}
\sum_{j=m}^{\mu-1}\sum_{k=n}^{v-1}\Delta_{22}c_{jk}&=\sum_{j=m}^{\mu-1}\sum_{k=n}^{v-1}\left(c_{jk}-c_{j+2,k}-\left(c_{j,k+2}-c_{j+2,k+2}\right)\right)\\
&=\sum_{k=n}^{v-1}\left(c_{mk}-c_{\mu k}-c_{\mu+1,k}+c_{m+1,k}-c_{m,k+2}+c_{\mu,k+2}+c_{\mu+1,k+2}-c_{m+1,k+2}  \right)\nonumber\\
&=\sum_{k=n}^{v-1}\left(\left(\left(c_{mk}-c_{m,k+2}\right)+\left(c_{m+1,k}-c_{m+1,k+2}\right)\right)-\left(\left(c_{\mu k}-c_{\mu,k+2}\right)+\left(c_{\mu+1,k}-c_{\mu+1,k+2}\right)\right)\right)\nonumber\\
&=c_{mn}-c_{mv}-c_{m,v+1}+c_{m,n+1}+c_{m+1,n}-c_{m+1,v}-c_{m+1,v+1}+c_{m+1,n+1}-c_{\mu n}\nonumber\\
&+c_{\mu v}+c_{\mu,v+1}-c_{\mu,n+1}-c_{\mu+1,n}+c_{\mu+1,v}+c_{\mu+1,v+1}-c_{\mu+1,n+1}\nonumber
\end{align*}
and applying the inequality (\ref{3.1}), we get
\begin{align}
c_{mn}&=\sum_{j=m}^{\mu-1}\sum_{k=n}^{v-1}\Delta_{22}c_{jk}+c_{mv}+c_{m,v+1}-c_{m,n+1}-c_{m+1,n}+c_{m+1,v}+c_{m+1,v+1}-c_{m+1,n+1}\nonumber\\
&+c_{\mu n}-c_{\mu v}-c_{\mu,v+1}+c_{\mu,n+1}+c_{\mu+1,n}-c_{\mu+1,v}-c_{\mu+1,v+1}+c_{\mu+1,n+1}\nonumber\\
&\le\sum_{j=m}^{2m-1}\sum_{k=n}^{2n-1}|\Delta_{22}c_{jk}|+c_{mv}+c_{m,v+1}+c_{m+1,v}+c_{m+1,v+1}+c_{\mu n}+c_{\mu,n+1}+c_{\mu+1,n}+c_{\mu+1,n+1}\nonumber\\
&\le\frac{\mathcal{C}}{mn}\left(\sup_{{M+N}\ge b_3(m+n)} \sum_{j=M}^{2M}\sum_{k=N}^{2N}c_{jk}\right)\nonumber\\
&+c_{mv}+c_{m,v+1}+c_{m+1,v}+c_{m+1,v+1}+c_{\mu n}+c_{\mu,n+1}+c_{\mu+1,n}+c_{\mu+1,n+1}\textbf{.}
\label{3.9}
\end{align}
Adding up all inequalities in (\ref{3.9}) for $\mu=m+1,m+2,\ldots,2m$ and $v=n+1,n+2,\ldots,2n$ we obtain
\begin{align*}
\sum_{\mu=m+1}^{2m}\sum_{v=n+1}^{2n}c_{mn}&\le \sum_{\mu=m+1}^{2m}\sum_{v=n+1}^{2n}\frac{\mathcal{C}}{mn}\left(\sup_{{M+N}\ge b_3(m+n)} \sum_{j=M}^{2M}\sum_{k=N}^{2N}c_{jk}\right)\\
&+\sum_{\mu=m+1}^{2m}\sum_{v=n+1}^{2n}c_{mv}+\sum_{\mu=m+1}^{2m}\sum_{v=n+1}^{2n}c_{m,v+1}+\sum_{\mu=m+1}^{2m}\sum_{v=n+1}^{2n}c_{m+1,v}+\sum_{\mu=m+1}^{2m}\sum_{v=n+1}^{2n}c_{m+1,v+1}\\
&+\sum_{v=n+1}^{2n}\sum_{\mu=m+1}^{2m}c_{\mu n}+\sum_{v=n+1}^{2n}\sum_{\mu=m+1}^{2m}c_{\mu,n+1}+\sum_{v=n+1}^{2n}\sum_{\mu=m+1}^{2m}c_{\mu+1,n}+\sum_{v=n+1}^{2n}\sum_{\mu=m+1}^{2m}c_{\mu+1,n+1}
\end{align*}
\begin{align*}
&=\mathcal{C}\left(\sup_{{M+N}\ge b_3(m+n)} \sum_{j=M}^{2M}\sum_{k=N}^{2N}c_{jk}\right)+\sum_{\mu=m+1}^{2m}\sum_{v=n+1}^{2n}(c_{mv}+c_{m+1,v})+\sum_{\mu=m+1}^{2m}\sum_{v=n+1}^{2n}(c_{m,v+1}+c_{m+1,v+1})\\
&+\sum_{v=n+1}^{2n}\sum_{\mu=m+1}^{2m}(c_{\mu n}+c_{\mu,n+1})+\sum_{v=n+1}^{2n}\sum_{\mu=m+1}^{2m}(c_{\mu+1,n}+c_{\mu+1,n+1}).
\end{align*}
Analogously as in (\ref{3.7}) and (\ref{3.8}) we can obtain the following inequalities
\begin{align*}
c_{mv}+c_{m+1,v}&\le \frac{\mathcal{C}}{m}\left(\max_{b_1(m)\le M\le\lambda b_1(m)}\sum_{j=M}^{2M}c_{jv}\right)+c_{\mu v}+c_{\mu+1,v}\qquad\textrm{for any}\; v\in\mathbb{N}\;\textrm{and}\;\mu=m,\cdots,2m,
\end{align*}
\begin{align*}
c_{\mu n}+c_{\mu,n+1}&\le\frac{\mathcal{C}}{n}\left(\max_{b_2(n)\le N\le\lambda b_2(n)}\sum_{k=N}^{2N}c_{\mu k}\right)+c_{\mu v}+c_{\mu,v+1}\qquad\textrm{for any}\; \mu\in\mathbb{N}\;\textrm{and}\;v=n,\cdots,2n.
\end{align*}
Hence we get
\begin{align*}
mnc_{mn}&\le
\mathcal{C}\left(\sup_{{M+N}\ge b_3(m+n)} \sum_{j=M}^{2M}\sum_{k=N}^{2N}c_{jk}\right)\\
&+\sum_{v=n+1}^{2n}\left(\sum_{\mu=m+1}^{2m}\frac{\mathcal{C}}{m}\left(\max_{b_1(m)\le M\le\lambda b_1(m)}\sum_{j=M}^{2M}c_{jv}\right)+\sum_{\mu=m+1}^{2m}(c_{\mu v}+c_{\mu+1,v})\right)\\
&+\sum_{v=n+1}^{2n}\left(\sum_{\mu=m+1}^{2m}\frac{\mathcal{C}}{m}\left(\max_{b_1(m)\le M\le\lambda b_1(m)}\sum_{j=M}^{2M}c_{j,v+1}\right)+\sum_{\mu=m+1}^{2m}(c_{\mu, v+1}+c_{\mu+1,v+1})\right)\\
&+\sum_{\mu=m+1}^{2m}\left(\sum_{v=n+1}^{2n}\frac{\mathcal{C}}{n}\left(\max_{b_2(n)\le N\le\lambda b_2(n)}\sum_{k=N}^{2N}c_{\mu k}\right)+\sum_{v=n+1}^{2n}(c_{\mu v}+c_{\mu,v+1})\right)\\
&+\sum_{\mu=m+1}^{2m}\left(\sum_{v=n+1}^{2n}\frac{\mathcal{C}}{n}\left(\max_{b_2(n)\le N\le\lambda b_2(n)}\sum_{k=N}^{2N}c_{\mu+1, k}\right)+\sum_{v=n+1}^{2n}(c_{\mu+1, v}+c_{\mu+1,v+1})\right)\\
&\le\mathcal{C}\left(\sup_{{M+N}\ge b_3(m+n)} \sum_{j=M}^{2M}\sum_{k=N}^{2N}c_{jk}\right)+\mathcal{C}\sum_{j=b_1(m)}^{2\lambda b_1(m)}\sum_{k=n}^{2n}c_{jk}+\sum_{j=m}^{2m}\sum_{k=n}^{2n}c_{jk}+\sum_{j=m}^{2m}\sum_{k=n}^{2n}c_{j+1,k}\\
&+\mathcal{C}\sum_{j=b_1(m)}^{2\lambda b_1(m)}\sum_{k=n}^{2n}c_{j,k+1}+\sum_{j=m}^{2m}\sum_{k=n}^{2n}c_{j,k+1}+\sum_{j=m}^{2m}\sum_{k=n}^{2n}c_{j+1,k+1}+\\
&+\mathcal{C}\sum_{j=m}^{2m}\sum_{k=b_2(n)}^{2\lambda b_2(n)}c_{jk}+\sum_{j=m}^{2m}\sum_{k=n}^{2n}c_{jk}+\sum_{j=m}^{2m}\sum_{k=n}^{2n}c_{j,k+1}\\
&+\mathcal{C}\sum_{j=m}^{2m}\sum_{k=b_2(n)}^{2\lambda b_2(n)}c_{j+1,k}+\sum_{j=m}^{2m}\sum_{k=n}^{2n}c_{j+1,k}+\sum_{j=m}^{2m}\sum_{k=n}^{2n}c_{j+1,k+1}\\
&\le\mathcal{C}\left(\sup_{{M+N}\ge b_3(m+n)} \sum_{j=M}^{2M}\sum_{k=N}^{2N}c_{jk}\right)+2\mathcal{C}\sum_{j=b_1(m)}^{2\lambda b_1(m)}\sum_{k=n}^{2n+1}c_{jk}+2\mathcal{C}\sum_{j=m}^{2m+1}\sum_{k=b_2(n)}^{2\lambda b_2(n)}c_{jk}+8\sum_{j=m}^{2m+1}\sum_{k=n}^{2n+1}c_{jk}\textbf{.}
\end{align*}
This ends the proof.
\end{proof} 

Denote, for $r\in\mathbb{N}$ and $k=0,1,2,\ldots,$ by 
\begin{align*}
\widetilde D_{k,r}(x)=\frac{\cos(k+\frac{r}{2})x}{2\sin(\frac{r}{2}x)}
\end{align*}
the conjugated Dirichlet type kernel.

\begin{lem}
(\cite{S}, \cite{2S}) Let $r\in \mathbb{N},\;l\in \mathbb{Z}$ and $%
\{a_{k}\}_{k=1}^{\infty }\subset \mathbb{C}$. If $x\neq \frac{2l\pi }{r}$,
then for all $m\geq n$ 
\begin{equation*}
\sum_{k=n}^{m}a_{k}\sin kx=-\sum_{k=n}^{m}\Delta _{r}a_{k}\widetilde{D}%
_{k,r}(x)+\sum_{k=m+1}^{m+r}a_{k}\widetilde{D}_{k,-r}(x)-%
\sum_{k=n}^{n+r-1}a_{k}\widetilde{D}_{k,-r}(x)
\end{equation*}%
\label{L.4}
\end{lem}


\begin{lem}
Let $\{c_{jk}\}_{j,k=1}^{\infty}\subset\mathbb{C}$ and $m,M,n,N\in\mathbb{N}$
such that $m\le M$ and $n\le N$.

\begin{description}
\item[(i)] If $x\in \left( 0,\frac{\pi }{2}\right) $, then 
\begin{equation}
\left\vert \sum_{j=m}^{M}c_{jk}\sin jx\right\vert \leq \frac{\pi }{4x}\left(
\sum_{j=m}^{M}|\Delta
_{20}c_{jk}|+\sum_{j=M+1}^{M+2}|c_{jk}|+\sum_{j=m}^{m+1}|c_{jk}|\right)
\label{3.10}
\end{equation}%
and if $x\in \left( \frac{\pi }{2},\pi \right) $, then 
\begin{equation*}
\left\vert \sum_{j=m}^{M}c_{jk}\sin jx\right\vert \leq \frac{\pi }{4(\pi -x)}%
\left( \sum_{j=m}^{M}|\Delta
_{20}c_{jk}|+\sum_{j=M+1}^{M+2}|c_{jk}|+\sum_{j=m}^{m+1}|c_{jk}|\right)
\end{equation*}%
for any $k\in \mathbb{N}$.

\item[(ii)] If $y\in \left( 0,\frac{\pi }{2}\right) $, then 
\begin{equation}
\left\vert \sum_{k=n}^{N}c_{jk}\sin ky\right\vert \leq \frac{\pi }{4y}\left(
\sum_{k=n}^{N}|\Delta
_{02}c_{jk}|+\sum_{k=N+1}^{N+2}|c_{jk}|+\sum_{k=n}^{n+1}|c_{jk}|\right)
\label{3.11}
\end{equation}%
and if $y\in \left( \frac{\pi }{2},\pi \right) $, then 
\begin{equation*}
\left\vert \sum_{k=n}^{N}c_{jk}\sin ky\right\vert \leq \frac{\pi }{4(\pi -y)}%
\left( \sum_{k=n}^{N}|\Delta
_{02}c_{jk}|+\sum_{k=N+1}^{N+2}|c_{jk}|+\sum_{k=n}^{n+1}|c_{jk}|\right)
\end{equation*}%
for any $j\in \mathbb{N}$.
\end{description}

\label{L.5}
\end{lem}

\begin{proof}
\begin{description}
\item[Part (i).] By Lemma \ref{L.4}, we have
\begin{align*}
\left|\sum_{j=m}^{M}c_{jk}\sin jx\right|&=\left|-\sum_{j=m}^{M}\Delta_{20}c_{jk}\widetilde D_{j,2}(x)+\sum_{j=M+1}^{M+2}c_{jk}\widetilde D_{j,-2}(x)-\sum_{j=m}^{m+1}c_{jk}\widetilde D_{j,-2}(x)\right|\\
&\le\sum_{j=m}^{M}|\Delta_{20}c_{jk}|\cdot|\widetilde D_{j,2}(x)|+\sum_{j=M+1}^{M+2}|c_{jk}|\cdot|\widetilde D_{j,-2}(x)|+\sum_{j=m}^{m+1}|c_{jk}|\cdot|\widetilde D_{j,-2}(x)|\textbf{.}
\end{align*}
If $x\in(0,\frac{\pi}{2})$, then using inequality $\sin x\ge\frac{2}{\pi}x$ we obtain the following estimation:
\begin{align}
\left|\widetilde D_{j,\pm 2}(x)\right|\le\left|\frac{\cos(j\pm 1)x}{2\sin(\pm x)}\right|\le\frac{1}{2\sin x}\le\frac{1}{\frac{4}{\pi}x}\le\frac{\pi}{4x}\textbf{.}\label{3.12}
\end{align}
From this we get
\begin{align*}
\left|\sum_{j=m}^{M}c_{jk}\sin jx\right|&\le\frac{\pi}{4x}\left(\sum_{j=m}^{M}|\Delta_{20}c_{jk}|+\sum_{j=M+1}^{M+2}|c_{jk}|+\sum_{j=m}^{m+1}|c_{jk}|\right)\textbf{.}
\end{align*}
If $x\in(\frac{\pi}{2},\pi)$, then using inequality $\sin x\ge 2-\frac{2}{\pi}x$ we have the estimation:
\begin{align}
\left|\widetilde D_{j,\pm 2}(x)\right|\le\frac{1}{2(2-\frac{2}{\pi}x)}\le\frac{\pi}{4(\pi-x)}\label{3.13}
\end{align}
and consequently
\begin{align*}
\left|\sum_{j=m}^{M}c_{jk}\sin jx\right|&\le\frac{\pi}{4(\pi-x)}\left(\sum_{j=m}^{M}|\Delta_{20}c_{jk}|+\sum_{j=M+1}^{M+2}|c_{jk}|+\sum_{j=m}^{m+1}|c_{jk}|\right)\textbf{.}
\end{align*}
\item[Part (ii).] Analogously as above
\begin{align*}
\left|\sum_{k=n}^{N}c_{jk}\sin ky\right|&=\left|-\sum_{k=n}^{N}\Delta_{02}c_{jk}\widetilde D_{k,2}(y)+\sum_{k=N+1}^{N+2}c_{jk}\widetilde D_{k,-2}(y)-\sum_{k=n}^{n+1}c_{jk}\widetilde D_{k,-2}(y)\right|\\
&\le\sum_{k=n}^{N}|\Delta_{02}c_{jk}|\cdot|\widetilde D_{k,2}(y)|+\sum_{k=N+1}^{N+2}|c_{jk}|\cdot|\widetilde D_{k,-2}(y)|+\sum_{k=n}^{n+1}|c_{jk}|\cdot|\widetilde D_{k,-2}(y)|\\
&\le\frac{\pi}{4y}\left(\sum_{k=n}^{N}|\Delta_{02}c_{jk}|+\sum_{k=N+1}^{N+2}|c_{jk}|+\sum_{k=n}^{n+1}|c_{jk}|\right)
\end{align*}
for $y\in\left(0,\frac{\pi}{2}\right)$, and
\begin{align*}
\left|\sum_{k=n}^{N}c_{jk}\sin ky\right|&\le\frac{\pi}{4(\pi-y)}\left(\sum_{k=n}^{N}|\Delta_{02}c_{jk}|+\sum_{k=N+1}^{N+2}|c_{jk}|+\sum_{k=n}^{n+1}|c_{jk}|\right)
\end{align*}
for $y\in\left(\frac{\pi}{2},\pi\right)$.
\newline
\end{description}

This ends the proof.
\end{proof}

\section{Main results}

We have the following results:

\begin{twr}
\qquad

\begin{description}
\item[(i)] If a double sequence $\{c_{jk}\}_{j,k=1}^{\infty}\subset\mathbb{C}
$ belongs to $DGM(_3\alpha,_3\beta,_3\gamma,2)$ and (\ref{2.2}) holds, then
the regular convergence of double sine series (\ref{2.1}) is uniform in ($%
x,y $).

\item[(ii)] Conversely, if a double sequence $\{c_{jk}\}_{j,k=1}^{\infty}%
\subset\mathbb{R}_+$ belongs to $DGM(_2\alpha,_2\beta,_2\gamma,2)$ and
double sine series (\ref{2.1}) is uniformly regularly convergent in ($x,y$),
then (\ref{2.2}) is satisfied.
\end{description}

\label{T.7}
\end{twr}

\begin{twr}
\qquad

\begin{description}
\item[(i)] $DGM(_3\alpha,_3\beta,_3\gamma,1)\subset
DGM(_3\alpha,_3\beta,_3\gamma,2)$.

\item[(ii)] There exists a double sequence $\{c_{jk}\}_{j,k=1}^{\infty}$,
with the property (\ref{2.2}), which belongs to the class $%
DGM(_3\alpha,_3\beta,_3\gamma,2)$ but it does not belong to the class $%
DGM(_3\alpha,_3\beta,_3\gamma,1)$.
\end{description}

\label{T.8}
\end{twr}

Analogously as in Theorem \ref{T.8}, we can show:

\begin{cor}
\qquad

\begin{description}
\item[(i)] $DGM(_2\alpha,_2\beta,_2\gamma,1)\subset
DGM(_2\alpha,_2\beta,_2\gamma,2)$.

\item[(ii)] There exists a double sequence $\{c_{jk}\}_{j,k=1}^{\infty}$,
with the property (\ref{2.2}), which belongs to the class $%
DGM(_2\alpha,_2\beta,_2\gamma,2)$ but it does not belong to the class $%
DGM(_2\alpha,_2\beta,_2\gamma,1)$.
\end{description}

\label{C.1}
\end{cor}

Now, we formulate some remarks.

\begin{rem}
From Theorem \ref{T.7}, using Theorem \ref{T.8} and Corollary \ref{C.1}, we
obtain Theorem \ref{T.5} and Theorem \ref{T.4}.
\end{rem}

\begin{rem}
There exist $(x_0,y_0)\in\mathbb{R}^2$ and a sequence $\{c_{jk}\}_{j,k=1}^{%
\infty}$ belonging to the class $DGM(_3\alpha,_3\beta,_3\gamma,3)$, with the
property (\ref{2.2}), such that the series (\ref{2.1}) is divergent in $%
(x_0,y_0)$. \label{R.2}
\end{rem}

\begin{rem}
Remark \ref{R.2} shows that the results from Theorem \ref{T.7} are not true
with $r=3$ instead of $r=2$.
\end{rem}

\section{Proof of the main results}

In this section we shall prove our main results.

\subsection{Proof of Theorem \protect\ref{T.7}}

\begin{description}
\item[Part (i):] Analogously as in (\cite{2S}, Theorem 2.5) we can show that
single series: 
\begin{equation}
\sum_{j=1}^{\infty }c_{jn}\sin jx,\qquad n=1,2,\ldots ,\qquad
\sum_{k=1}^{\infty }c_{mk}\sin ky,\qquad m=1,2,\ldots  \label{5.1}
\end{equation}%
are uniformly convergent since $\{c_{jn}\}_{j=1}^{\infty }\in GM(_{3}\alpha
,2)$ for any $n\in \mathbb{N}$ and $\{c_{mk}\}_{k=1}^{\infty }\in
GM(_{3}\beta ,2)$ for any $m\in \mathbb{N}$. Let $\epsilon >0$ be
arbitrarily fixed. We shall prove that for any $M\geq m>\eta ,$ $N\geq
n>\eta $ and any $(x,y)\in \mathbb{R}^{2}$ we have 
\begin{equation}
\left\vert \sum_{j=m}^{M}\sum_{k=n}^{N}c_{jk}\sin jx\sin ky\right\vert
<(1+2\pi \mathcal{C}+2\pi +1,5\pi ^{2}\mathcal{C}+\pi ^{2})\epsilon \mathbf{,%
}  \label{5.2}
\end{equation}%
where $\eta =\eta (\epsilon )>\lambda $ is the natural number which
satisfies for any $m,n>\eta $ 
\begin{equation*}
mn|c_{mn}|<\epsilon ,\qquad mn\sum_{j=m}^{\infty }\sum_{k=n}^{\infty
}|\Delta _{22}c_{jk}|<16C\epsilon ,\qquad m\sum_{j=m}^{\infty }\sup_{k\geq
n}k|\Delta _{20}c_{jk}|<4C\epsilon ,\qquad n\sum_{k=n}^{\infty }\sup_{j\geq
m}j|\Delta _{02}c_{jk}|<4C\epsilon .
\end{equation*}%
The inequality (\ref{5.2}) is trivial, when $x=0$ and $y$ is arbitrary or $%
y=0$ and $x$ is arbitrary. We have the same situation, if $x=\pi $ and $y$
is arbitrary or $y=\pi $ and $x$ is arbitrary. Suppose $x,y\in \left( 0,%
\frac{\pi }{2}\right) $, set $\mu :=\left\lceil \frac{1}{x}\right\rceil $
and $v:=\left\lceil \frac{1}{y}\right\rceil $, where $\lceil \cdot \rceil $
means the integer part of a real number. We have four cases:

\begin{description}
\item {\textsc{Case} (a):} $\eta <m\leq M<\mu \;$ and $\;\eta <n\leq N<v$.
Using the inequality $\sin x\leq x$, we have 
\begin{equation*}
\left\vert \sum_{j=m}^{M}\sum_{k=n}^{N}c_{jk}\sin jx\sin ky\right\vert \leq
xy\sum_{j=m}^{M}\sum_{k=n}^{N}jk|c_{jk}|<\frac{1}{\mu v}\sum_{j=m}^{\mu
}\sum_{k=n}^{v}\epsilon \leq \epsilon \mathbf{.}
\end{equation*}

\item {\textsc{Case} (b):} $\;\max {\{\eta ,\mu \}}<m\leq M\;$ and $\;\eta
<n\leq N<v$. We obtain 
\begin{equation*}
\left\vert \sum_{j=m}^{M}\sum_{k=n}^{N}c_{jk}\sin jx\sin ky\right\vert \leq
\sum_{k=n}^{N}\left\vert \sin ky\right\vert \cdot \left\vert
\sum_{j=m}^{M}c_{jk}\sin jx\right\vert \leq y\sum_{k=n}^{N}k\left\vert
\sum_{j=m}^{M}c_{jk}\sin jx\right\vert \mathbf{.}
\end{equation*}%
By (\ref{3.10}) 
\begin{align*}
\left\vert \sum_{j=m}^{M}\sum_{k=n}^{N}c_{jk}\sin jx\sin ky\right\vert &
\leq y\sum_{k=n}^{N}k\frac{\pi }{4x}\left( \sum_{j=m}^{M}|\Delta
_{20}c_{jk}|+\sum_{j=M+1}^{M+2}|c_{jk}|+\sum_{j=m}^{m+1}|c_{jk}|\right) \\
& \leq y\sum_{k=n}^{N}k\frac{\pi \mu }{4}\left( \sum_{j=m}^{M}|\Delta
_{20}c_{jk}|+\sum_{j=M+1}^{M+2}|c_{jk}|+\sum_{j=m}^{m+1}|c_{jk}|\right) \\
& \leq \frac{\pi }{4v}\sum_{k=n}^{v}\left( m\sup_{k\geq
n}k\sum_{j=m}^{\infty }|\Delta _{20}c_{jk}|+4\sup_{j\geq m}\sup_{k\geq
n}jk|c_{jk}|\right)
\end{align*}%
and using Lemma \ref{L.2} and (\ref{2.2}) we get 
\begin{equation*}
\left\vert \sum_{j=m}^{M}\sum_{k=n}^{N}c_{jk}\sin jx\sin ky\right\vert <%
\frac{\pi }{4}(4\mathcal{C}\epsilon +4\epsilon )=(\pi \mathcal{C}+\pi
)\epsilon \mathbf{.}
\end{equation*}

\item {\textsc{Case} (c):} $\;\eta <m\leq M<\mu \;$ and $\;\max {\{\eta ,v\}}%
<n\leq N$. We have 
\begin{equation*}
\left\vert \sum_{j=m}^{M}\sum_{k=n}^{N}c_{jk}\sin jx\sin ky\right\vert \leq
\sum_{j=m}^{M}|\sin jx|\cdot \left\vert \sum_{k=n}^{N}c_{jk}\sin
ky\right\vert \leq x\sum_{j=m}^{M}j\left\vert \sum_{k=n}^{N}c_{jk}\sin
ky\right\vert \mathbf{.}
\end{equation*}%
Using (\ref{3.11}) 
\begin{align*}
\left\vert \sum_{j=m}^{M}\sum_{k=n}^{N}c_{jk}\sin jx\sin ky\right\vert &
\leq x\sum_{j=m}^{M}j\frac{\pi }{4y}\left( \sum_{k=n}^{N}|\Delta
_{02}c_{jk}|+\sum_{k=N+1}^{N+2}|c_{jk}|+\sum_{k=n}^{n+1}|c_{jk}|\right) \\
& \leq x\sum_{j=m}^{M}j\frac{\pi v}{4}\left( \sum_{k=n}^{N}|\Delta
_{02}c_{jk}|+\sum_{k=N+1}^{N+2}|c_{jk}|+\sum_{k=n}^{n+1}|c_{jk}|\right) \\
& \leq \frac{\pi }{4\mu }\sum_{j=m}^{\mu }\left( n\sup_{j\geq
m}j\sum_{k=n}^{\infty }|\Delta _{02}c_{jk}|+4\sup_{j\geq m}\sup_{k\geq
n}jk|c_{jk}|\right)
\end{align*}%
and by Lemma \ref{L.2} and (\ref{2.2}) 
\begin{equation*}
\left\vert \sum_{j=m}^{M}\sum_{k=n}^{N}c_{jk}\sin jx\sin ky\right\vert <%
\frac{\pi }{4}\left( 4\mathcal{C}\epsilon +4\epsilon \right) =(\pi \mathcal{C%
}+\pi )\epsilon \mathbf{.}
\end{equation*}

\item {\textsc{Case} (d):} $\;\max {\{\eta ,\mu \}}<m\leq M\;$ and $\;\max {%
\{\eta ,v\}}<n\leq N$. Using Lemma \ref{L.4}, we get

\begin{align*}
\left\vert \sum_{j=m}^{M}\sum_{k=n}^{N}\right. & \left. c_{jk}\sin jx\sin
ky\right\vert \leq \left\vert \sum_{j=m}^{M}\left( -\sum_{k=n}^{N}\Delta
_{02}c_{jk}\widetilde{D}_{k,2}(y)+\sum_{k=N+1}^{N+2}c_{jk}\widetilde{D}%
_{k,-2}(y)-\sum_{k=n}^{n+1}c_{jk}\widetilde{D}_{k,-2}(y)\right) \sin
jx\right\vert \\
& =\left\vert -\sum_{k=n}^{N}\left( \sum_{j=m}^{M}\Delta _{02}c_{jk}\sin
jx\right) \widetilde{D}_{k,2}(y)+\sum_{k=N+1}^{N+2}\left(
\sum_{j=m}^{M}c_{jk}\sin jx\right) \widetilde{D}_{k,-2}(y)\right. \\
& \left. -\sum_{k=n}^{n+1}\left( \sum_{j=m}^{M}c_{jk}\sin jx\right) 
\widetilde{D}_{k,-2}(y)\right\vert \\
& =\left\vert -\sum_{k=n}^{N}\left( -\sum_{j=m}^{M}\Delta _{02}(\Delta
_{20}c_{jk})\widetilde{D}_{j,2}(x)+\sum_{j=M+1}^{M+2}\Delta _{02}c_{jk}%
\widetilde{D}_{j,-2}-\sum_{j=m}^{m+1}\Delta _{02}c_{jk}\widetilde{D}%
_{j,-2}(x)\right) \widetilde{D}_{k,2}(y)\right. \\
& +\sum_{k=N+1}^{N+2}\left( -\sum_{j=m}^{M}\Delta _{20}c_{jk}\widetilde{D}%
_{j,2}(x)+\sum_{j=M+1}^{M+2}c_{jk}\widetilde{D}_{j,-2}-\sum_{j=m}^{m+1}c_{jk}%
\widetilde{D}_{j,-2}(x)\right) \widetilde{D}_{k,-2}(y) \\
& \left. -\sum_{k=n+1}^{n+1}\left( -\sum_{j=m}^{M}\Delta _{20}c_{jk}%
\widetilde{D}_{j,2}(x)+\sum_{j=M+1}^{M+2}c_{jk}\widetilde{D}%
_{j,-2}-\sum_{j=m}^{m+1}c_{jk}\widetilde{D}_{j,-2}(x)\right) \widetilde{D}%
_{k,-2}(y)\right\vert \\
& \leq \sum_{j=m}^{M}\sum_{k=n}^{N}|\Delta _{22}c_{jk}|\cdot |\widetilde{D}%
_{j,2}(x)|\cdot |\widetilde{D}_{k,2}(y)|+\sum_{j=M+1}^{M+2}\sum_{k=n}^{N}|%
\Delta _{02}c_{jk}|\cdot |\widetilde{D}_{j,-2}(x)|\cdot |\widetilde{D}%
_{k,2}(y)| \\
& +\sum_{j=m}^{m+1}\sum_{k=n}^{N}|\Delta _{02}c_{jk}|\cdot |\widetilde{D}%
_{j,-2}(x)|\cdot |\widetilde{D}_{k,2}(y)|+\sum_{j=m}^{M}\sum_{k=N+1}^{N+2}|%
\Delta _{20}c_{jk}|\cdot |\widetilde{D}_{j,2}(x)|\cdot |\widetilde{D}%
_{k,-2}(y)| \\
& +\sum_{j=M+1}^{M+2}\sum_{k=N+1}^{N+2}|c_{jk}|\cdot |\widetilde{D}%
_{j,-2}(x)|\cdot |\widetilde{D}_{k,-2}(y)|+\sum_{j=m}^{m+1}%
\sum_{k=N+1}^{N+2}|c_{jk}|\cdot |\widetilde{D}_{j,-2}(x)|\cdot |\widetilde{D}%
_{k,-2}(y)| 
\end{align*}
\begin{align*}
& +\sum_{j=m}^{M}\sum_{k=n}^{n+1}|\Delta _{20}c_{jk}|\cdot |\widetilde{D}%
_{j,2}(x)|\cdot |\widetilde{D}_{k,-2}(y)|+\sum_{j=M+1}^{M+2}%
\sum_{k=n}^{n+1}|c_{jk}|\cdot |\widetilde{D}_{j,-2}(x)|\cdot |\widetilde{D}%
_{k,-2}(y)| \\
& +\sum_{j=m}^{m+1}\sum_{k=n}^{n+1}|c_{jk}|\cdot |\widetilde{D}%
_{j,-2}(x)|\cdot |\widetilde{D}_{k,-2}(y)|.
\end{align*}%
Using (\ref{3.12}), (\ref{2.2}) Lemma \ref{L.1} and Lemma \ref{L.2}, we
obtain 
\begin{align*}
\left\vert \sum_{j=m}^{M}\right. & \left. \sum_{k=n}^{N}c_{jk}\sin jx\sin
ky\right\vert \leq \\
& \leq \frac{\pi ^{2}}{16xy}\left( \sum_{j=m}^{M}\sum_{k=n}^{N}|\Delta
_{22}c_{jk}|+\sum_{j=M+1}^{M+2}\sum_{k=n}^{N}|\Delta
_{02}c_{jk}|+\sum_{j=m}^{m+1}\sum_{k=n}^{N}|\Delta
_{02}c_{jk}|+\sum_{j=m}^{M}\sum_{k=N+1}^{N+2}|\Delta _{20}c_{jk}|\right. \\
& \left.
+\sum_{j=M+1}^{M+2}\sum_{k=N+1}^{N+2}|c_{jk}|+\sum_{j=m}^{m+1}%
\sum_{k=N+1}^{N+2}|c_{jk}|+\sum_{j=m}^{M}\sum_{k=n}^{n+1}|\Delta
_{20}c_{jk}|+\sum_{j=M+1}^{M+2}\sum_{k=n}^{n+1}|c_{jk}|+\sum_{j=m}^{m+1}%
\sum_{k=n}^{n+1}|c_{jk}|\right) \\
& \leq \frac{\pi ^{2}}{16}\left( mn\sum_{j=m}^{\infty }\sum_{k=n}^{\infty
}|\Delta _{22}c_{jk}|+4m\sup_{k\geq n}k\sum_{j=m}^{\infty }|\Delta
_{20}c_{jk}|+4n\sup_{j\geq m}j\sum_{k=n}^{\infty }|\Delta
_{02}c_{jk}|+16\sup_{j\geq m}\sup_{k\geq n}jk|c_{jk}|\right) \\
& <\frac{\pi ^{2}}{16}(16\mathcal{C}\epsilon +4\mathcal{C}\epsilon +4%
\mathcal{C}\epsilon +16\epsilon )=\left( \frac{3}{2}\mathcal{C}+1\right) \pi
^{2}\epsilon .
\end{align*}
\end{description}

\qquad

Let $x\in \left( \frac{\pi }{2},\pi \right) $ and $y\in \left( 0,\frac{\pi }{%
2}\right) $, set $\mu :=\left\lceil \frac{1}{\pi -x}\right\rceil $ and $%
v:=\left\lceil \frac{1}{y}\right\rceil $. We have four cases:

\begin{description}
\item {\textsc{Case} ($a^{\ast }$):} $\;\eta <m\leq M<\mu \;$ and $\;\eta
<n\leq N<v$. Using the inequality $\sin x\leq \pi -x$ and $\sin y\leq y$, we
get 
\begin{equation*}
\left\vert \sum_{j=m}^{M}\sum_{k=n}^{N}c_{jk}\sin jx\sin ky\right\vert \leq
(\pi -x)y\sum_{j=m}^{M}\sum_{k=n}^{N}jk|c_{jk}|<\frac{1}{\mu v}%
\sum_{j=m}^{\mu }\sum_{k=n}^{v}\epsilon \leq \epsilon \mathbf{.}
\end{equation*}

\item {\textsc{Case} ($b^{\ast }$):} $\;\max {\{\eta ,\mu \}}<m\leq M\;$ and 
$\;\eta <n\leq N<v$. Applying (\ref{3.13}) we get 
\begin{align*}
\left\vert \sum_{j=m}^{M}\sum_{k=n}^{N}c_{jk}\sin jx\sin ky\right\vert &
\leq \sum_{k=n}^{N}\left\vert \sin ky\right\vert \cdot \left\vert
\sum_{j=m}^{M}c_{jk}\sin jx\right\vert \leq y\sum_{k=n}^{N}k\left\vert
\sum_{j=m}^{M}c_{jk}\sin jx\right\vert \\
& \leq y\sum_{k=n}^{N}k\frac{\pi }{4(\pi -x)}\left( \sum_{j=m}^{M}|\Delta
_{20}c_{jk}|+\sum_{j=M+1}^{M+2}|c_{jk}|+\sum_{j=m}^{m+1}|c_{jk}|\right)
<(\pi \mathcal{C}+\pi )\epsilon \mathbf{.}
\end{align*}

\item {\textsc{Case} ($c^{\ast }$):} $\;\eta <m\leq M<\mu \;$ and $\;\max {%
\{\eta ,v\}}<n\leq N$. Analogously as in case (c), we have the inequality 
\begin{equation*}
\left\vert \sum_{j=m}^{M}\sum_{k=n}^{N}c_{jk}\sin jx\sin ky\right\vert \leq
\sum_{j=m}^{M}|\sin jx|\cdot \left\vert \sum_{k=n}^{N}c_{jk}\sin
ky\right\vert \leq (\pi -x)\sum_{j=m}^{M}j\left\vert
\sum_{k=n}^{N}c_{jk}\sin ky\right\vert <(\pi \mathcal{C}+\pi )\epsilon 
\mathbf{.}
\end{equation*}

\item {\textsc{Case} ($d^{\ast }$):} $\;\max {\{\eta ,\mu \}}<m\leq M\;$ and 
$\;\max {\{\eta ,v\}}<n\leq N$. Using (\ref{3.12}) and (\ref{3.13})
analogously as in case~(d), we obtain 
\begin{align*}
\left\vert \sum_{j=m}^{M}\right. & \left. \sum_{k=n}^{N}c_{jk}\sin jx\sin
ky\right\vert \leq \\
& \leq \frac{\pi ^{2}}{16(\pi -x)y}\left(
\sum_{j=m}^{M}\sum_{k=n}^{N}|\Delta
_{22}c_{jk}|+\sum_{j=M+1}^{M+2}\sum_{k=n}^{N}|\Delta
_{02}c_{jk}|+\sum_{j=m}^{m+1}\sum_{k=n}^{N}|\Delta
_{02}c_{jk}|+\sum_{j=m}^{M}\sum_{k=N+1}^{N+2}|\Delta _{20}c_{jk}|\right. \\
& \left.
+\sum_{j=M+1}^{M+2}\sum_{k=N+1}^{N+2}|c_{jk}|+\sum_{j=m}^{m+1}%
\sum_{k=N+1}^{N+2}|c_{jk}|+\sum_{j=m}^{M}\sum_{k=n}^{n+1}|\Delta
_{20}c_{jk}|+\sum_{j=M+1}^{M+2}\sum_{k=n}^{n+1}|c_{jk}|+\sum_{j=m}^{m+1}%
\sum_{k=n}^{n+1}|c_{jk}|\right) \\
& <\left( \frac{3}{2}\mathcal{C}+1\right) \pi ^{2}\epsilon .
\end{align*}
\end{description}

Let $x\in \left( 0,\frac{\pi }{2}\right) $ and $y\in \left( \frac{\pi }{2}%
,\pi \right) $, set $\mu :=\left\lceil \frac{1}{x}\right\rceil $ and $%
v:=\left\lceil \frac{1}{\pi -y}\right\rceil $. Now, we have also four cases:

\begin{description}
\item {\textsc{Case} ($a^{\ast \ast }$):} $\eta <m\leq M<\mu \;$ and $\;\eta
<n\leq N<v$. Using the inequality $\sin x\leq x$ and $\sin y\leq \pi -y$, we
get 
\begin{equation*}
\left\vert \sum_{j=m}^{M}\sum_{k=n}^{N}c_{jk}\sin jx\sin ky\right\vert \leq
x(\pi -y)\sum_{j=m}^{M}\sum_{k=n}^{N}jk|c_{jk}|<\frac{1}{\mu v}%
\sum_{j=m}^{\mu }\sum_{k=n}^{v}\epsilon \leq \epsilon \mathbf{.}
\end{equation*}

\item {\textsc{Case} ($b^{\ast \ast }$):} $\;\max {\{\eta ,\mu \}}<m\leq M\;$
and $\;\eta <n\leq N<v$. We obtain similarly as in case (b) 
\begin{equation*}
\left\vert \sum_{j=m}^{M}\sum_{k=n}^{N}c_{jk}\sin jx\sin ky\right\vert \leq
\sum_{k=n}^{N}\left\vert \sin ky\right\vert \cdot \left\vert
\sum_{j=m}^{M}c_{jk}\sin jx\right\vert \leq (\pi
-y)\sum_{k=n}^{N}k\left\vert \sum_{j=m}^{M}c_{jk}\sin jx\right\vert <(\pi 
\mathcal{C}+\pi )\epsilon \mathbf{.}
\end{equation*}

\item {\textsc{Case} ($c^{\ast \ast }$):} $\;\eta <m\leq M<\mu \;$ and $%
\;\max {\{\eta ,v\}}<n\leq N$. Applying (\ref{3.13}) we get 
\begin{align*}
\left\vert \sum_{j=m}^{M}\sum_{k=n}^{N}c_{jk}\sin jx\sin ky\right\vert &
\leq \sum_{j=m}^{M}|\sin jx|\cdot \left\vert \sum_{k=n}^{N}c_{jk}\sin
ky\right\vert \leq x\sum_{j=m}^{M}j\left\vert \sum_{k=n}^{N}c_{jk}\sin
ky\right\vert \\
& \leq x\sum_{j=m}^{M}j\frac{\pi }{4(\pi -y)}\left( \sum_{k=n}^{N}|\Delta
_{02}c_{jk}|+\sum_{k=N+1}^{N+2}|c_{jk}|+\sum_{k=n}^{n+1}|c_{jk}|\right)
<(\pi \mathcal{C}+\pi )\epsilon \mathbf{.}
\end{align*}

\item {\textsc{Case} ($d^{\ast \ast }$):} $\;\max {\{\eta ,\mu \}}<m\leq M\;$
and $\;\max {\{\eta ,v\}}<n\leq N$. Using (\ref{3.12}) and (\ref{3.13}) and
analogously as in case (d), we get 
\begin{align*}
\left\vert \sum_{j=m}^{M}\right. & \left. \sum_{k=n}^{N}c_{jk}\sin jx\sin
ky\right\vert \leq \\
& \leq \frac{\pi ^{2}}{16x(\pi -y)}\left(
\sum_{j=m}^{M}\sum_{k=n}^{N}|\Delta
_{22}c_{jk}|+\sum_{j=M+1}^{M+2}\sum_{k=n}^{N}|\Delta
_{02}c_{jk}|+\sum_{j=m}^{m+1}\sum_{k=n}^{N}|\Delta
_{02}c_{jk}|+\sum_{j=m}^{M}\sum_{k=N+1}^{N+2}|\Delta _{20}c_{jk}|\right. \\
& \left.
+\sum_{j=M+1}^{M+2}\sum_{k=N+1}^{N+2}|c_{jk}|+\sum_{j=m}^{m+1}%
\sum_{k=N+1}^{N+2}|c_{jk}|+\sum_{j=m}^{M}\sum_{k=n}^{n+1}|\Delta
_{20}c_{jk}|+\sum_{j=M+1}^{M+2}\sum_{k=n}^{n+1}|c_{jk}|+\sum_{j=m}^{m+1}%
\sum_{k=n}^{n+1}|c_{jk}|\right) \\
& <\left( \frac{3}{2}\mathcal{C}+1\right) \pi ^{2}\epsilon .
\end{align*}
\end{description}

Finally, let $x\in \left( \frac{\pi }{2},\pi \right) $ and $y\in \left( 
\frac{\pi }{2},\pi \right) $, set $\mu :=\left\lceil \frac{1}{\pi -x}%
\right\rceil $ and $v:=\left\lceil \frac{1}{\pi -y}\right\rceil $.
Analogously as before we have four cases:

\begin{description}
\item {\textsc{Case} ($a^{\ast \ast \ast }$):} $\;\eta <m\leq M<\mu \;$ and $%
\;\eta <n\leq N<v$. Using the inequality $\sin x\leq \pi -x$ and $\sin y\leq
\pi -y$, we get 
\begin{equation*}
\left\vert \sum_{j=m}^{M}\sum_{k=n}^{N}c_{jk}\sin jx\sin ky\right\vert \leq
(\pi -x)(\pi -y)\sum_{j=m}^{M}\sum_{k=n}^{N}jk|c_{jk}|<\frac{1}{\mu v}%
\sum_{j=m}^{\mu }\sum_{k=n}^{v}\epsilon \leq \epsilon \mathbf{.}
\end{equation*}

\item {\textsc{Case} ($b^{\ast \ast \ast }$):} $\;\max {\{\eta ,\mu \}}%
<m\leq M\;$ and $\;\eta <n\leq N<v$. We have similarly as in case ($b^{\ast
} $) 
\begin{equation*}
\left\vert \sum_{j=m}^{M}\sum_{k=n}^{N}c_{jk}\sin jx\sin ky\right\vert \leq
(\pi -y)\sum_{k=n}^{N}k\left\vert \sum_{j=m}^{M}c_{jk}\sin jx\right\vert
<(\pi \mathcal{C}+\pi )\epsilon \mathbf{.}
\end{equation*}

\item {\textsc{Case} ($c^{\ast \ast \ast }$):} $\;\eta <m\leq M<\mu \;$ and $%
\;\max {\{\eta ,v\}}<n\leq N$. We obtain similarly as in case ($c^{\ast \ast
}$) 
\begin{equation*}
\left\vert \sum_{j=m}^{M}\sum_{k=n}^{N}c_{jk}\sin jx\sin ky\right\vert \leq
\leq (\pi -x)\sum_{j=m}^{M}j\left\vert \sum_{k=n}^{N}c_{jk}\sin
ky\right\vert <(\pi \mathcal{C}+\pi )\epsilon \mathbf{.}
\end{equation*}

\item {\textsc{Case} ($d^{\ast \ast \ast }$):} $\;\max {\{\eta ,\mu \}}%
<m\leq M\;$ and $\;\max {\{\eta ,v\}}<n\leq N$. Using (\ref{3.12}) and (\ref%
{3.13}), we have 
\begin{align*}
& \left\vert \sum_{j=m}^{M}\sum_{k=n}^{N}c_{jk}\sin jx\sin ky\right\vert \leq
\\
& \leq \frac{\pi ^{2}}{16(\pi -x)(\pi -y)}\left(
\sum_{j=m}^{M}\sum_{k=n}^{N}|\Delta
_{22}c_{jk}|+\sum_{j=M+1}^{M+2}\sum_{k=n}^{N}|\Delta
_{02}c_{jk}|+\sum_{j=m}^{m+1}\sum_{k=n}^{N}|\Delta
_{02}c_{jk}|+\sum_{j=m}^{M}\sum_{k=N+1}^{N+2}|\Delta _{20}c_{jk}|\right. \\
& \left.
+\sum_{j=M+1}^{M+2}\sum_{k=N+1}^{N+2}|c_{jk}|+\sum_{j=m}^{m+1}%
\sum_{k=N+1}^{N+2}|c_{jk}|+\sum_{j=m}^{M}\sum_{k=n}^{n+1}|\Delta
_{20}c_{jk}|+\sum_{j=M+1}^{M+2}\sum_{k=n}^{n+1}|c_{jk}|+\sum_{j=m}^{m+1}%
\sum_{k=n}^{n+1}|c_{jk}|\right) \\
& <\left( \frac{3}{2}\mathcal{C}+1\right) \pi ^{2}\epsilon
\end{align*}
\end{description}

If we summarize all partial estimations we get (\ref{5.2}), this ends the
proof of part (i).\hspace{\fill}
\end{description}

\begin{description}
\item[Part (ii):] Suppose that $\{c_{jk}\}_{j,k=1}^{\infty }$ is
non-negative and let $\epsilon >0$ be arbitrarily fixed. Using the form (\ref%
{2.2}) for the uniform regular convergence of (\ref{2.1}), we find that
there exists an integer $m_{0}=m_{0}(\epsilon )$ for which 
\begin{equation}
\left\vert \sum_{j=m}^{M}\sum_{k=n}^{N}c_{jk}\sin jx\sin ky\right\vert
<\epsilon  \label{5.3}
\end{equation}%
holds for any $m+n>m_{0}$ and any ($x,y$). Set $x_{1}(m)=\frac{\pi }{4m}$, $%
x_{2}(m)=\frac{\pi }{4\lambda b_1(m)}$, $y_{1}(n)=\frac{\pi }{4n}$, $y_{2}(n)=%
\frac{\pi }{4\lambda b_2(n)}$\qquad we have 
\begin{align}
& \sin (jx_{1}(m))\geq \sin \frac{\pi }{4}\qquad \text{if}\qquad m\leq j\leq
2m+1;\qquad \sin (jx_{2}(m))\geq \sin \frac{\pi }{4\lambda }\qquad \text{if}%
\qquad b_1(m)\leq j\leq 2\lambda b_1(m);  \notag \\
& \sin (ky_{1}(n))\geq \sin \frac{\pi }{4}\qquad \text{if}\qquad n\leq k\leq
2n+1;\qquad \sin (ky_{2}(n))\geq \sin \frac{\pi }{4\lambda }\qquad \text{if}%
\qquad b_2(n)\leq k\leq 2\lambda b_2(n).  \notag
\end{align}%
Since $\{b_{1}(l)\}_{l=1}^{\infty }$, $\{b_{2}(l)\}_{l=1}^{\infty }$, $%
\{b_{3}(l)\}_{l=1}^{\infty }$ tends to infinity, there exists an $m_{1}$
such that for any $m$, $n$: $m+n>m_{1}$ implies $m+n>m_{0}$, $%
b_{1}(m)+n>m_{0}$, $m+b_{2}(n)>m_{0}$ and $b_{3}(m+n)>m_{0}$. Then by (\ref%
{5.3}) and Lemma \ref{L.3}, we have for $m+n>m_{1}$

\begin{align*}
(5\mathcal{C}+8)\epsilon & >\mathcal{C}\left( \sup_{M+N\geq
b_{3}(m+n)}\sum_{j=M}^{2M}\sum_{k=N}^{2N}c_{jk}\sin (jx_{1}(M))\sin
(ky_{1}(N))\right) \\
& +2\mathcal{C}\sum_{j=b_{1}(m)}^{2\lambda
b_{1}(m)}\sum_{k=n}^{2n+1}c_{jk}\sin (jx_{2}(m))\sin (ky_{1}(n))+2\mathcal{C}%
\sum_{j=m}^{2m+1}\sum_{k=b_{2}(n)}^{2\lambda b_{2}(n)}c_{jk}\sin
(jx_{1}(m))\sin (ky_{2}(n)) \\
& +8\sum_{j=m}^{2m+1}\sum_{k=n}^{2n+1}c_{jk}\sin (jx_{1}(m))\sin (ky_{1}(n))%
\mathbf{.}
\end{align*}%
Next 
\begin{align*}
(5\mathcal{C}+8)\epsilon & >\mathcal{C}\left( \sin \frac{\pi }{4}\right)
^{2}\left( \sup_{{M+N}\geq b_{3}(m+n)}\sum_{j=M}^{2M}\sum_{k=N}^{2N}{|c_{jk}|%
}\right) +2\mathcal{C}\sin \frac{\pi }{4\lambda }\sin \frac{\pi }{4}%
\sum_{j=b_{1}(m)}^{2\lambda b_{1}(m)}\sum_{k=n}^{2n+1}c_{jk} \\
& +2\mathcal{C}\sin \frac{\pi }{4}\sin \frac{\pi }{4\lambda }%
\sum_{j=m}^{2m+1}\sum_{k=b_{2}(n)}^{2\lambda b_{2}(n)}c_{jk}+8\left( \sin 
\frac{\pi }{4}\right) ^{2}\sum_{j=m}^{2m+1}\sum_{k=n}^{2n+1}c_{jk}
\end{align*}%
and finally, we have 
\begin{equation*}
(5\mathcal{C}+8)\epsilon >\left( \sin \frac{\pi }{4}\sin \frac{\pi }{%
4\lambda }\right) mnc_{mn}\qquad \text{whenever}\qquad m+n>m_{1}:\text{and}%
:m,n>\lambda .
\end{equation*}%
Hence (\ref{2.2}) is satisfied when $j+k\rightarrow \infty $ and $j,k\geq
\lambda $. If $j\rightarrow \infty $ and $k<\lambda $ or $j<\lambda $ and $%
k\rightarrow \infty $, (\ref{2.2}) follows from the uniform convergence of
the series in (\ref{5.1}). It completes the proof of part (ii).\hspace{\fill}
$\square $
\end{description}

\subsection{Proof of Theorem \protect\ref{T.8}}

\begin{description}
\item[Part (i):] We prove that $DGM(_{3}\alpha ,_{3}\beta ,_{3}\gamma
,1)\subset DGM(_{3}\alpha ,_{3}\beta ,_{3}\gamma ,2)$. Let $%
\{c_{jk}\}_{j,k=1}^{\infty }\in DGM(_{3}\alpha ,_{3}\beta ,_{3}\gamma ,1)$.
It easy to see that 
\begin{align*}
\sum_{j=m}^{2m-1}|\Delta _{20}c_{jn}|& \leq \sum_{j=m}^{\infty }|\Delta
_{20}c_{jn}|=\sum_{j=m}^{\infty }|c_{jn}-c_{j+1,n}+c_{j+1,n}-c_{j+2,n}| \\
& \leq \sum_{j=m}^{\infty }|c_{jn}-c_{j+1,n}|+\sum_{j=m}^{\infty
}|c_{j+1,n}-c_{j+2,n}|\leq 2\sum_{j=m}^{\infty
}|c_{jn}-c_{j+1,n}|=2\sum_{j=m}^{\infty }|\Delta _{10}c_{jn}|
\end{align*}%
and by (\ref{2.9}), for $m\geq \lambda $ and $n\geq 1$, 
\begin{align*}
\sum_{j=m}^{\infty }|\Delta _{10}c_{jn}|& =\sum_{r=0}^{\infty
}\sum_{j=2^{r}m}^{2^{r+1}m-1}|\Delta _{10}c_{jn}|\leq \sum_{r=0}^{\infty }%
\frac{\mathcal{C}}{2^{r}m}\left( \sup_{M\geq
b(2^{r}m)}\sum_{j=M}^{2M}|c_{jn}|\right) \\
& \leq \sum_{r=0}^{\infty }\frac{1}{2^{r}}\left( \frac{\mathcal{C}}{m}%
\sup_{M\geq b(m)}\sum_{j=M}^{2M}|c_{jn}|\right) =2\cdot \frac{\mathcal{C}}{m}%
\sup_{M\geq b(m)}\sum_{j=M}^{2M}|c_{jn}|\mathbf{.}
\end{align*}%
Next, we get 
\begin{equation}
\sum_{j=m}^{2m-1}|\Delta _{20}c_{jn}|\leq 4\cdot \frac{\mathcal{C}}{m}%
\sup_{M\geq b(m)}\sum_{j=M}^{2M}|c_{jn}|\mathbf{.}  \label{5.4}
\end{equation}%
Similarly as above, for $n\geq \lambda $ and $m\geq 1$, 
\begin{equation}
\sum_{k=n}^{2n-1}|\Delta _{02}c_{mk}|\leq 4\cdot \frac{\mathcal{C}}{n}%
\sup_{N\geq b(n)}\sum_{k=N}^{2N}|c_{mk}|\mathbf{.}  \label{5.5}
\end{equation}%
Now, we have 
\begin{align*}
\sum_{j=m}^{2m-1}\sum_{k=n}^{2n-1}|\Delta _{22}c_{jk}|&
=\sum_{j=m}^{2m-1}\sum_{k=n}^{2n-1}|\Delta _{11}c_{jk}+\Delta
_{11}c_{j+1,k}+\Delta _{11}c_{j,k+1}+\Delta _{11}c_{j+1,k+1}| \\
& \leq \sum_{j=m}^{\infty }\sum_{k=n}^{\infty }|\Delta
_{11}c_{jk}|+\sum_{j=m+1}^{\infty }\sum_{k=n}^{\infty }|\Delta
_{11}c_{jk}|+\sum_{j=m}^{\infty }\sum_{k=n+1}^{\infty }|\Delta
_{11}c_{jk}|+\sum_{j=m+1}^{\infty }\sum_{k=n+1}^{\infty }|\Delta _{11}c_{jk}|
\\
& \leq 4\sum_{j=m}^{\infty }\sum_{k=n}^{\infty }|\Delta _{11}c_{jk}|
\end{align*}%
and by (\ref{2.11}), we obtain for $m,n\geq \lambda $ 
\begin{align*}
\sum_{j=m}^{\infty }\sum_{k=n}^{\infty }|\Delta _{11}c_{jk}|&
=\sum_{r=0}^{\infty }\sum_{s=0}^{\infty
}\sum_{j=2^{r}m}^{2^{r+1}m-1}\sum_{j=2^{s}n}^{2^{s+1}n-1}|\Delta
_{11}c_{jk}|\leq \sum_{r=0}^{\infty }\sum_{s=0}^{\infty }\frac{\mathcal{C}}{%
2^{r}m\cdot 2^{s}n}\left( \sup_{M+N\geq
b(2^{r}m+2^{s}n)}\sum_{j=M}^{2M}\sum_{k=N}^{2N}|c_{jk}|\right) \\
& \leq 4\cdot \frac{\mathcal{C}}{mn}\left( \sup_{M+N\geq
b(m+n)}\sum_{j=M}^{2M}\sum_{k=N}^{2N}|c_{jk}|\right) \mathbf{.}
\end{align*}%
Finally, we get 
\begin{equation}
\sum_{j=m}^{2m-1}\sum_{k=n}^{2n-1}|\Delta _{22}c_{jk}|\leq 16\cdot \frac{%
\mathcal{C}}{mn}\left( \sup_{M+N\geq
b(m+n)}\sum_{j=M}^{2M}\sum_{k=N}^{2N}|c_{jk}|\right) \mathbf{.}  \label{5.6}
\end{equation}

From (\ref{5.4}), (\ref{5.5}) and (\ref{5.6}) we have that $%
\{c_{jk}\}_{j,k=1}^{\infty }\in DGM(_{3}\alpha ,_{3}\beta ,_{3}\gamma ,2)$.
\bigskip

\item[Part (ii):] Let 
\begin{equation}
c_{jk}=\frac{2+(-1)^{j}}{j^{2}}\cdot \frac{2+(-1)^{k}}{k^{2}}\qquad \text{for%
}\;j,k\in \mathbb{N}\mathbf{.}  \label{5.7}
\end{equation}%
We show that $\{c_{jk}\}_{j,k=1}^{\infty }\in DGM(_{3}\alpha ,_{3}\beta
,_{3}\gamma ,2)$. It easy to see that 
\begin{equation*}
\Delta _{20}c_{jk}=c_{jk}-c_{j+2,k}=c_{jk}\cdot \frac{4(j+1)}{(j+2)^{2}}
\end{equation*}%
and 
\begin{align*}
\sum_{j=m}^{2m-1}\left\vert \Delta _{20}c_{\text{jn}}\right\vert &
=\sum_{j=m}^{2m-1}\left\vert c_{jn}\cdot \frac{4(j+1)}{(j+2)^{2}}\right\vert
\leq \sum_{j=m}^{2m-1}|c_{jn}|\frac{4}{j+1}\leq 4\sum_{j=m}^{2m-1}|c_{jn}|%
\frac{1}{j}\leq \frac{4}{m}\sum_{j=m}^{2m-1}|c_{jn}|\leq \frac{4}{m}%
\sum_{j=m}^{2m}|c_{jn}| \\
& \leq \frac{4}{m}\sup_{M\geq b(m)}\sum_{j=M}^{2M}|c_{jn}|\mathbf{.}
\end{align*}%
Similarly as above 
\begin{equation*}
\Delta _{02}c_{jk}=c_{jk}-c_{j,k+2}=c_{jk}\cdot \frac{4(k+1)}{(k+2)^{2}}
\end{equation*}%
and 
\begin{equation*}
\sum_{k=n}^{2n-1}|\Delta _{02}c_{mk}|\leq \frac{4}{n}\sup_{N\geq
b(n)}\sum_{k=N}^{2N}|c_{mk}|\mathbf{.}
\end{equation*}%
By elementary calculations 
\begin{equation*}
\Delta _{22}c_{\text{jk}}=c_{\text{jk}%
}-c_{j+2,k-}c_{j,k+2}+c_{j+2,k+2}=c_{jk}\cdot \frac{4(j+1)}{(j+2)^{2}}\cdot 
\frac{4(k+1)}{(k+2)^{2}}
\end{equation*}%
and 
\begin{align*}
\sum_{j=m}^{2m-1}\sum_{k=n}^{2n-1}|\Delta _{22}c_{jk}|&
=\sum_{j=m}^{2m-1}\sum_{k=n}^{2n-1}\left\vert c_{jk}\cdot \frac{4(j+1)}{%
(j+2)^{2}}\cdot \frac{4(k+1)}{(k+2)^{2}}\right\vert \leq
16\sum_{j=m}^{2m-1}\sum_{k=n}^{2n-1}|c_{jk}|\cdot \frac{1}{j+1}\cdot \frac{1%
}{k+1} \\
& \leq 16\sum_{j=m}^{2m-1}\sum_{k=n}^{2n-1}|c_{jk}|\frac{1}{jk}\leq \frac{16%
}{mn}\sum_{j=m}^{2m-1}\sum_{k=n}^{2n-1}|c_{jk}|\leq \frac{16}{mn}%
\sum_{j=m}^{2m}\sum_{k=n}^{2n}|c_{jk}| \\
& \leq \frac{16}{mn}\sup_{M+N\geq
b(m+n)}\sum_{j=M}^{2M}\sum_{k=N}^{2N}|c_{jk}|\mathbf{.}
\end{align*}%
Therefore $\{c_{jk}\}_{j,k=1}^{\infty }\in DGM(_{3}\alpha ,_{3}\beta
,_{3}\gamma ,2)$. Now, we show that $\{c_{jk}\}_{j,k=1}^{\infty }\notin
DGM(_{3}\alpha ,_{3}\beta ,_{3}\gamma ,1)$. We have 
\begin{align*}
\sum_{j=m}^{2m-1}\left\vert \Delta _{10}c_{jn}\right\vert &
=\sum_{j=m}^{2m-1}\left\vert \frac{(-1)^{j}+2}{j^{2}}\cdot \frac{(-1)^{n}+2}{%
n^{2}}-\frac{(-1)^{j+1}+2}{(j+1)^{2}}\cdot \frac{(-1)^{n}+2}{n^{2}}%
\right\vert \\
& =\frac{(-1)^{n}+2}{n^{2}}\sum_{j=m}^{2m-1}\left\vert \frac{(-1)^{j}+2}{%
j^{2}}-\frac{2-(-1)^{j}}{(j+1)^{2}}\right\vert \mathbf{.}
\end{align*}%
Let $A_{m}=\{j:\;m\leq j\leq 2m-1\;\text{and}\;j\;\text{is even}\}$. Then 
\begin{equation*}
\sum_{j=m}^{2m-1}\left\vert \Delta _{10}c_{jn}\right\vert \geq \frac{1}{n^{2}%
}\sum_{j\in A_{m}}\left( \frac{3}{j^{2}}-\frac{1}{(j+1)^{2}}\right) \geq 
\frac{1}{n^{2}}\sum_{j\in A_{m}}\left( \frac{3}{j^{2}}-\frac{1}{j^{2}}%
\right) \geq \frac{2}{n^{2}}\sum_{j\in A_{m}}\frac{1}{j^{2}}\geq \frac{m-1}{%
2n^{2}m^{2}}
\end{equation*}%
and since 
\begin{align*}
\frac{\mathcal{C}}{m}\sup_{M\geq b(m)}\sum_{j=M}^{2M}|c_{jn}|& =\frac{%
\mathcal{C}}{m}\sup_{M\geq b(m)}\sum_{j=M}^{2M}\left\vert \frac{(-1)^{j}+2}{%
j^{2}}\cdot \frac{(-1)^{n}+2}{n^{2}}\right\vert \leq \frac{\mathcal{C}}{m}%
\frac{(-1)^{n}+2}{n^{2}}\sup_{M\geq b(m)}\sum_{j=M}^{2M}\left\vert \frac{%
(-1)^{j}+2}{j^{2}}\right\vert \\
& \leq \frac{3\mathcal{C}}{mn^{2}}\sup_{M\geq b(m)}\sum_{j=M}^{2M}\frac{3}{%
j^{2}}\leq \frac{9\mathcal{C}}{mn^{2}}\sup_{M\geq b(m)}\frac{2}{M}\leq \frac{%
18\mathcal{C}}{mn^{2}b(m)}
\end{align*}%
the inequality 
\begin{equation*}
\sum_{j=m}^{2m-1}\left\vert \Delta _{10}c_{jn}\right\vert \leq \frac{%
\mathcal{C}}{m}\sup_{M\geq b(m)}\sum_{j=M}^{2M}|c_{jn}|
\end{equation*}%
does not hold, because $\frac{1}{b(m)}\rightarrow 0$ as $m\rightarrow \infty 
$.
\end{description}

This ends the proof.\hspace{\stretch{1}} $\square$

\subsection{Proof of Corollary \protect\ref{C.1}}

\begin{description}
\item[Part (i):] We prove that $DGM(_{2}\alpha ,_{2}\beta ,_{2}\gamma
,1)\subset DGM(_{2}\alpha ,_{2}\beta ,_{2}\gamma ,2)$. Let $%
\{c_{jk}\}_{j,k=1}^{\infty }\in DGM(_{2}\alpha ,_{2}\beta ,_{2}\gamma ,1)$.
Then for $m\geq \lambda $ and $n\geq 1$

\begin{align*}
\sum_{j=m}^{2m-1}|\Delta _{20}c_{jn}|& =\sum_{j=m}^{2m-1}|\Delta
_{10}c_{jn}+\Delta _{10}c_{j+1,n}|\leq \sum_{j=m}^{2m-1}|\Delta
_{10}c_{jn}|+\sum_{j=m}^{2m-1}|\Delta _{10}c_{j+1,n}| \\
& \leq \sum_{j=m}^{2m-1}|\Delta _{10}c_{jn}|+\sum_{j=m+1}^{2m}|\Delta
_{10}c_{jn}|\leq 2\sum_{j=m}^{2m}|\Delta _{10}c_{jn}|+|\Delta _{10}c_{2m,n}|
\end{align*}
\begin{align*}
& \leq 2\frac{\mathcal{C}}{m}\max_{b_{1}(m)\leq M\leq \lambda
b_{1}(m)}\sum_{j=M}^{2M}c_{jn}+\sum_{k=2m}^{4m-1}|\Delta _{10}c_{jn}| \\
& \leq 2\frac{\mathcal{C}}{m}\left( \max_{b_{1}(m)\leq M\leq \lambda
b_{1}(m)}\sum_{j=M}^{2M}c_{jn}+\max_{b_{1}(2m)\leq M\leq \lambda
b_{1}(2m)}\sum_{j=M}^{2M}c_{jn}\right) =2\frac{\mathcal{C}}{m}\left\{
S_{1}+S_{2}\right\} .
\end{align*}%
Let, for $m\geq \lambda $, $b_{1}^{\prime }(m)=b_{1}\left( m\right) $ if $%
S_{2}\leq S_{1}$ or $b_{1}^{\prime }(m)=b_{1}\left( 2m\right) $ if $%
S_{1}\leq S_{2}$. Then 
\begin{equation}
\sum_{j=m}^{2m-1}|\Delta _{20}c_{jn}|\leq 4\frac{\mathcal{C}}{m}%
\max_{b_{1}^{\prime }(m)\leq M\leq \lambda b_{1}^{\prime
}(m)}\sum_{j=M}^{2m}c_{jn}  \label{5.8}
\end{equation}%
for $m\geq \lambda $ and $n\geq 1.$ Similarly as above 
\begin{equation}
\sum_{k=n}^{2n-1}|\Delta _{02}c_{mk}|\leq 4\frac{\mathcal{C}}{n}%
\max_{b_{2}^{\prime }(n)\leq N\leq \lambda b_{2}^{\prime
}(n)}\sum_{k=N}^{2n}c_{mk},  \label{5.9}
\end{equation}%
for $n\geq \lambda $ and $m\geq 1$. Finally, by (\ref{5.6}) we get 
\begin{equation}
\sum_{j=m}^{2m-1}\sum_{k=n}^{2n-1}|\Delta _{22}c_{jk}|\leq 16\cdot \frac{%
\mathcal{C}}{mn}\left( \sup_{M+N\geq
b_{3}(m+n)}\sum_{j=M}^{2M}\sum_{k=N}^{2N}|c_{jk}|\right)  \label{5.10}
\end{equation}%
for $m,n\geq \lambda $. From (\ref{5.8}), (\ref{5.9}) and (\ref{5.10}) we
get that $\{c_{jk}\}_{j,k=1}^{\infty }\in DGM(_{2}\alpha ,_{2}\beta
,_{2}\gamma ,2)$. \newline

\item[Part (ii):] Taking the sequence (\ref{5.7}), we can show, similarly as
in the proof of Theorem \ref{T.8} (part (ii)) that $\{c_{jk}\}_{j,k=1}^{%
\infty }\in DGM(_{2}\alpha ,_{2}\beta ,_{2}\gamma ,2)$ and $%
\{c_{jk}\}_{j,k=1}^{\infty }\notin DGM(_{2}\alpha ,_{2}\beta ,_{1}\gamma ,1)$%
.
\end{description}

It completes the proof.\hspace{\stretch{1}} $\square$

\subsection{Proof of Remark \protect\ref{R.2}}

Let $c_{jk}=a_{j}\cdot a_{k}$, where 
\begin{equation*}
a_{n}=\left\{ 
\begin{array}{ccc}
\frac{3}{n\ln (n+1)} & \text{if} & n=3l+1, \\ 
\frac{1}{n\ln (n+1)} & \text{if} & n\neq 3l+1,%
\end{array}%
\right.
\end{equation*}%
for $l\in 
\mathbb{N}
\cup \left\{ 0\right\} $ and $(x_{0},y_{0})=\left( \frac{2}{3}\pi ,\frac{2%
}{3}\pi \right) $.\newline

It is easy to see that (\ref{2.2}) holds. Now, we shall prove that $%
\{c_{jk}\}_{j,k=1}^{\infty }\in DGM(_{3}\alpha ,_{3}\beta ,_{3}\gamma ,3)$.
Let $\qquad $%
\begin{equation*}
A_{m}=\{j:\;j=3l+1,\;l\in 
\mathbb{N}
\cup \left\{ 0\right\},\;m\leq j\leq 2m-1\}%
\end{equation*}%
and $\qquad $%
\begin{equation*}
B_{m}=\{j:\;j\neq 3l+1,\;l\in 
\mathbb{N}
\cup \left\{ 0\right\},\;m\leq j\leq 2m-1\}.
\end{equation*}%
Then we have%
\begin{equation*}
\sum_{j=m}^{2m-1}|\Delta _{30}c_{jn}|=\sum_{j=m}^{2m-1}a_{n}|a_{j}-a_{j+3}|
\end{equation*}%
\begin{align*}
& =\sum_{j\in A_{m}}a_{n}\left( \frac{3}{j\ln (j+1)}-\frac{3}{(j+3)\ln (j+4)}%
\right) +\sum_{j\in B_{m}}a_{n}\left( \frac{1}{j\ln (j+1)}-\frac{1}{(j+3)\ln
(j+4)}\right)  \\
& =3\sum_{j\in A_{m}}a_{n}\left( \frac{j(\ln (j+4)-\ln (j+1))+3\ln (j+4)}{%
j(j+3)\ln (j+1)\ln (j+4)}\right)  \\
& +\sum_{j\in B_{m}}a_{n}\left( \frac{j(\ln (j+4)-\ln (j+1))+3\ln (j+4)}{%
j(j+3)\ln (j+1)\ln (j+4)}\right) \mathbf{.}
\end{align*}%
Applying the Lagrange theorem, for the function $y=\ln (x)$, there exists $%
c\in (j+4,j+1)$ such that 
\begin{equation*}
\ln (j+4)-\ln (j+1)=\frac{3}{c}\leq \frac{3}{j+1}
\end{equation*}%
and 
\begin{align}
\sum_{j=m}^{2m-1}|\Delta _{30}c_{jn}|& \leq 3\sum_{j\in A_{m}}a_{n}\left( 
\frac{\frac{3j}{j+1}+3\ln (j+4)}{j(j+3)\ln (j+1)\ln (j+4)}\right)
+\sum_{j\in B_{m}}a_{n}\left( \frac{\frac{3j}{j+1}+3\ln (j+4)}{j(j+3)\ln
(j+1)\ln (j+4)}\right)   \notag \\
& \leq 3\sum_{j\in A_{m}}a_{n}\left( \frac{3+3\ln (j+4)}{j^{2}\ln (j+1)\ln
(j+4)}\right) +\sum_{j\in B_{m}}a_{n}\left( \frac{3+3\ln (j+4)}{j^{2}\ln
(j+1)\ln (j+4)}\right)   \notag \\
& \leq 3\sum_{j\in A_{m}}a_{n}\left( \frac{6}{j^{2}\ln (j+1)}\right)
+\sum_{j\in B_{m}}a_{n}\left( \frac{6}{j^{2}\ln (j+1)}\right) =6\sum_{j\in
A_{m}}a_{n}a_{j}\frac{1}{j}+6\sum_{j\in B_{m}}a_{n}a_{j}\frac{1}{j}  \notag
\\
& \leq \frac{6}{m}\sum_{j=m}^{2m}|c_{jn}|\leq \frac{6}{m}\sup_{M\geq
b(m)}\sum_{j=M}^{2M}|c_{jn}|  \label{5.11}
\end{align}%
Similarly as above 
\begin{equation}
\sum_{k=n}^{2n-1}|\Delta _{03}c_{mk}|\leq \frac{6}{n}\sup_{N\geq
b(n)}\sum_{k=N}^{2n}|c_{mk}|\mathbf{.}  \label{5.12}
\end{equation}%
By elementary calculations we have 
\begin{align}
&\sum_{j=m}^{2m-1}\sum_{k=n}^{2n-1}|\Delta
_{33}c_{jk}|=\sum_{j=m}^{2m-1}%
\sum_{k=n}^{2n-1}|(a_{j}-a_{j+3})(a_{k}-a_{k+3})|=\sum_{j=m}^{2m-1}\left%
\vert a_{j}-a_{j+3}\right\vert \sum_{k=n}^{2n-1}\left\vert
a_{k}-a_{k+3}\right\vert \notag \\
& =\left( \sum_{j\in A_{m}}\left( \frac{3}{j\ln (j+1)}-\frac{3}{(j+3)\ln
(j+4)}\right) +\sum_{j\in B_{m}}\left( \frac{1}{j\ln (j+1)}-\frac{1}{%
(j+3)\ln (j+4)}\right) \right)  \notag \\
& \cdot \left( \sum_{k\in A_{n}}\left( \frac{3}{k\ln (k+1)}-\frac{3}{%
(k+3)\ln (k+4)}\right) +\sum_{k\in B_{n}}\left( \frac{1}{k\ln (k+1)}-\frac{1%
}{(k+3)\ln (k+4)}\right) \right)  \notag \\
& \leq \left( 6\sum_{j\in A_{m}}\frac{3}{j^{2}\ln (j+1)}+6\sum_{j\in B_{m}}%
\frac{1}{j^{2}\ln (j+1)}\right) \left( 6\sum_{k\in A_{n}}\frac{3}{k^{2}\ln
(k+1)}+6\sum_{k\in B_{n}}\frac{1}{k^{2}\ln (k+1)}\right)  \notag \\
& =\left( 6\sum_{j\in A_{m}}c_{j}\frac{1}{j}+6\sum_{j\in B_{m}}c_{j}\frac{1}{%
j}\right) \left( 6\sum_{k\in A_{n}}c_{k}\frac{1}{k}+6\sum_{k\in B_{n}}c_{k}%
\frac{1}{k}\right) \leq \frac{36}{mn}\sum_{j=m}^{2m-1}%
\sum_{k=n}^{2n-1}|c_{jk}|  \notag \\
& \leq \frac{36}{mn}\sup_{M+N\geq
b(m+n)}\sum_{j=M}^{2M}\sum_{k=N}^{2N}|c_{jk}|\mathbf{.}  \label{5.13}
\end{align}%
From (\ref{5.11}), (\ref{5.12}) and (\ref{5.13}) we obtain that $%
\{c_{jk}\}_{j,k=1}^{\infty }\in DGM(_{3}\alpha ,_{3}\beta ,_{3}\gamma ,3)$.%
\newline

Moreover, we have 
\begin{align*}
& \sum_{j=1}^{3M+2}\sum_{k=1}^{3N+2}c_{jk}\sin jx_{0}\sin ky_{0} \\
& =\left( a_{1}\sin \frac{2}{3}\pi +a_{2}\sin \frac{4}{3}\pi
+\sum_{j=3}^{3M+2}a_{j}\sin jx_{0}\right) \left( a_{1}\sin \frac{2}{3}\pi
+a_{2}\sin \frac{4}{3}\pi +\sum_{k=3}^{3N+2}a_{k}\sin ky_{0}\right) \\
& =\left[ \left( a_{1}\sin \frac{2}{3}\pi -a_{2}\sin \frac{2}{3}\pi \right)
+\sum_{j=1}^{M}\sum_{i=0}^{2}a_{3j+i}\sin \left( (3j+i)\frac{2}{3}\pi
\right) \right] \\
& \cdot \left[ \left( a_{1}\sin \frac{2}{3}\pi -a_{2}\sin \frac{2}{3}\pi
\right) +\sum_{k=1}^{N}\sum_{i=0}^{2}a_{3k+i}\sin \left( (3k+i)\frac{2}{3}%
\pi \right) \right] \\
& =\left[ \sin \frac{2}{3}\pi (a_{1}-a_{2})+\sum_{j=1}^{M}\left(
a_{3j+1}\sin \frac{2}{3}\pi +a_{3j+2}\sin \frac{4}{3}\pi \right) \right]
\cdot \left[ \sin \frac{2}{3}\pi (a_{1}-a_{2})+\sum_{k=1}^{N}\left(
a_{3k+1}\sin \frac{2}{3}\pi +a_{3k+2}\sin \frac{4}{3}\pi \right) \right] \\
& =\left( \sin \frac{2}{3}\pi \right) ^{2}\sum_{j=0}^{M}\sum_{k=0}^{N}\left( 
\frac{3}{(3j+1)\ln (3j+2)}-\frac{1}{(3j+2)\ln (3j+3)}\right) \cdot \left( 
\frac{3}{(3k+1)\ln (3k+2)}-\frac{1}{(3k+2)\ln (3k+3)}\right) \\
& \geq \left( \sin \frac{2}{3}\pi \right) ^{2}\sum_{j=0}^{M}\sum_{k=0}^{N}%
\frac{2}{(3j+1)\ln (3j+3)}\cdot \frac{2}{(3k+1)\ln (3k+3)}\rightarrow \infty
\qquad \text{as}\qquad M+N\rightarrow \infty \mathbf{.}
\end{align*}

This ends the proof.\hspace{\stretch{1}} $\square$

\end{document}